\documentclass[a4paper,12pt,draft]{article}
\usepackage{a4wide}

\usepackage[nompar,nosign]{commenting}
\declareauthor{lo}{Luke}{red}
\authorcommand{lo}{comment}
\declareauthor{mv}{Matthijs}{blue}
\authorcommand{mv}{comment}

\setdefaultauthor{lo}

\makeatletter
\renewcommand{\comm@todo@mpar}[1]{}
\makeatother

\def\divider{%
  \leavevmode\leaders\hrule height 0.6ex depth \dimexpr0.4pt-0.6ex\hfill%
  \kern0pt%
}

\renewcommand{\OpenCommentBraket}{$\lceil$}
\renewcommand{\ClosedCommentBraket}{$\rfloor$}

\usepackage{paralist}

\RequirePackage{xcolor}
\RequirePackage{hyperref}
\hypersetup{final}
\hypersetup{%
    bookmarksopen,
    colorlinks,
    linkcolor=blue,
    citecolor=brown,
    backref,
    urlcolor=blue,
    pdfstartview=FitH}
\makeatletter
\g@addto@macro\UrlBreaks{\do\-}

\def\donotbreak{\@beginparpenalty=10000}

\makeatother

\usepackage{amsmath,amssymb}
\usepackage{mathtools}
\usepackage{stmaryrd}
\usepackage{calculation}

\newcommand\set[1]{{\{ }{#1}{\} }}

\newcommand\pair[2]{\left<{#1}, {#2}\right>}
\newcommand\reals{\mathbb{R}}

\newcommand\naturals{\mathbb{N}}

\newcommand\dom{\mathsf{dom}}
\newcommand\fst{\mathsf{fst}}
\newcommand\snd{\mathsf{snd}}

\newcommand\dcompose{\mathbin{;}}
\newcommand\id[1][{}]{\mathrm{id}_{#1}}
\newcommand\Category[1]{\ensuremath{\mathbf{#1}}}
\newcommand\Set{\Category{Set}}

\newcommand\dif{\mathrm{d}}
\newcommand{\Meas}{\Category{Meas}}
\newcommand\Sbs{\Category{Sbs}}
\newcommand\counting{\#}
\newcommand\lebesgue{\mathbb{U}}
\newcommand\uniform[1]{\mathbb{U}_{#1}}
\newcommand\lli{\stackrel{\infty}{\ll}}
\newcommand\boti{<\mkern-9mu\infty\;}
\newcommand\sample{\mathbf{sample}}
\newcommand\score{\mathbf{score}}
\newcommand\dett{\mathsf{det}}

\newcommand\floor{\mathsf{floor}}
\newcommand\Qbs{\Category{Qbs}}
\newcommand\mutsing[2]{{#1} \,\, \bot\,\, {#2}}
\newcommand\kact[2]{{#1}\triangleright{#2}}
\newcommand\letrecin[3]{\mathbf{letrec}\,#1=#2\,\mathbf{in}\,#3}
\newcommand\ifelse[3]{\mathbf{if}\,#1\,\mathbf{then}\,#2\,\mathbf{else}\,#3}
\newcommand\letin[3]{\mathbf{let}\,#1=#2\,\mathbf{in}\,#3}

\usepackage{amsthm}

\theoremstyle{plain}
\newtheorem{lemma}{Lemma}
\newtheorem{theorem}{Theorem}
\newtheorem{proposition}{Proposition}
\newtheorem{corollary}{Corollary}

\theoremstyle{definition}
\newtheorem{definition}{Definition}
\newtheorem{example}{Example}

\newtheorem*{correctness-assumption}{Correctness Assumption}

\theoremstyle{remark}
\newtheorem{remark}{Remark}

\newenvironment{proof*}
    {\begin{proof}[Proof \textup{(Sketch)}]}
    {\end{proof}}

\usepackage{tikz-cd}

\usepackage[square]{natbib}

\usepackage{datetime}

\title{On S-Finite Measures and Kernels}
\author{Matthijs V\'ak\'ar and Luke Ong}
\date{}

\begin{document}

\maketitle

\ddmmyyyydate
\renewcommand{\dateseparator}{-}
\newcommand{\Null}{\mathcal{N}}

%[Work in progress! \texttt{Begun September 2017. Draft of \today, \currenttime}]
\begin{abstract}
In this note, we develop some of the basic theory of s-finite (measures and)
kernels, a little-studied class that Staton has recently argued convincingly
to be precisely the semantic counterpart of (first-order) probabilistic programs.
We discuss their Carath\'eodory extension and
extend Staton's analysis of their product measures.
We give various characterisations of such kernels and discuss
their relationship to the more commonly studied classes of $\sigma$-finite, subprobability
and probability kernels. 
We use these characterisations to establish suitable Radon-Nikod\'ym, Lebesgue
decomposition and disintegration theorems for s-finite kernels.
We discuss s-finite analogues of the classical randomisation lemma for probability
kernels.
Throughout, we give some examples to explain the connection with (first-order)
probabilistic programming.
Finally, we briefly explore how some of these results extend to quasi-Borel spaces,
and hence how they apply to higher-order probabilistic programming. 
\end{abstract}

\clearpage
\tableofcontents
\clearpage

\section{Introduction}
With increasing computational power and increasingly large datasets  
available, probabilistic computation is becoming increasingly tractable and  
interesting.
As society increasingly depends on probabilistic 
algorithms, for instance through various safety-critical machine learning 
applications such as self-driving cars, it becomes more pressing to 
guarantee the correctness of such algorithms.
At the same time, it is notoriously difficult to achieve good test coverage 
for programs that involve probabilistic branching. 
This suggests that proof-based correctness arguments could be of particular 
value in this domain.

A foundational question that needs to be answered for these purposes is what 
the semantics of a probabilistic program should be.
In particular, what program transformations (for instance, compiler 
optimisations or inference algorithms) on probabilistic programs should be 
considered semantics-preserving, and thus safe?

The goal of probabilistic programming is to allow users to specify
statistical models and to interpret data in the light of these models by
performing (approximate) Bayesian inference.
The key idea is to separate the specification of the model from that of the
inference algorithm.
A probabilistic programming language supplies general-purpose inference
algorithms that work for larger classes of models, so that the user does
not have to manually implement inference for each model she writes, freeing
her to put more energy into accurate modelling.

Intuitively, the semantics of a statistical model is a certain unnormalised
measure, in the case of closed programs (which do not take any inputs), or a
certain unnormalised (Markov) kernel, in the case of open programs (which take
inputs).
The idea then is that the semantics of an inference algorithm is a procedure
that attempts to (approximately) normalise such a measure representing a
(closed) model.
In practice, this may be done by directly computing an approximate normalised
distribution or by drawing samples from the approximate normalised distribution;
the normalising constant may or may not be computed in the inference algorithm.

To make this more precise, we need to specify what class of measures
statistical models defined by probabilistic programs correspond to.
In practice, infinite measures like the Lebesgue measure are often used
in programs as improper priors.
Moreover, as argued in \citep{staton2017commutative}, infinite measures are in fact
unavoidable in any probabilistic programming language with probability
distributions over natural or real numbers and soft constraints, as they can be
encoded through importance sampling.
This introduces the challenge of finding a suitable class of possibly infinite
measures and kernels that is closed under composition.
Importantly, the commonly studied class of $\sigma$-finite measures and kernels
is not closed under composition!
\cite{staton2017commutative} has recently argued convincingly that closed
(first-order) probabilistic programs correspond precisely to s-finite measures
and kernels, a strictly larger class of unnormalised measures and kernels.

This result, emphasising the importance of s-finite measures and kernels, is
particularly striking because they have hardly been studied in the past, as
probabilists have focused their energy on the more limited $\sigma$-finite class.
Therefore, much of the foundational theory in the s-finite setting remains to be
established.

We hope to make a contribution to such a theory in this note.
In particular, we give various characterisations of s-finite measures and
kernels and prove Radon-Nikod\'ym, Lebesgue decomposition, disintegration and
randomisation theorems.
Throughout, we take a more compositional approach than is conventionally taken,
focusing on kernels rather than mere measures where possible, in the hope that
the results will ultimately be applicable to a modular analysis of probabilistic
programs.

For instance, as we shall see, the Radon-Nikod\'ym theorem gives the precise
conditions under which transformations of probabilistic programs known as
importance sampling and rejection sampling are valid, the disintegration theorem
gives the theoretical foundation for when exact Bayesian inference through
symbolic disintegration as described in \citep{shan2017exact} is possible,
and the randomisation lemma demonstrates how any s-finite kernel
(hence any probabilistic program) can be implemented as a combination of
a single random number generator with some reweighting followed by a pure
deterministic program.

\begin{remark}
NB: In this paper, we use the convention that $\infty\cdot 0=0$ and $
\infty\cdot r=\infty$ for any other $r\in (0,\infty]$.
Similarly, we shall use the convention that $r/\infty=0$ for any
$r\in [0,\infty)$, while $\infty/\infty$ is undefined.
\end{remark}

\section{Recap: Measures and Kernels}
We give a very brief recap of the basic definitions of measure theory. 
A \emph{measurable space} $(|X|,\Sigma_X)$ (often we simply write $X$ for the 
pair) is a set $|X|$ equipped with a countably complete Boolean subalgebra 
$\Sigma_X\subseteq \mathcal{P}|X|$ of its power set, called the
\emph{$\sigma$-algebra} of measurable subsets. 
Every countable set is a measurable space in a canonical way, using the 
\emph{discrete $\sigma$-algebra}: every subset is measurable. 
More generally, every topological space $X$ has a canonical $\sigma$-algebra 
(called the Borel $\sigma$-algebra) which is generated by its open 
sets\footnote{
This can be thought of as the smallest classical logic generated by the 
intuitionistic logic represented by the (open sets of the) topological 
space.}. 
This gives a canonical $\sigma$-algebra on, for instance, $\reals$ and 
$[0,\infty]$.

A \emph{measurable function} $f:X\to Y$ is a function from $|X|$ to $|Y|$ such 
that $V\in \Sigma_Y\Rightarrow f^{-1}(V)\in\Sigma_X$. 
We write $\Meas$ for the \emph{category of measurable spaces and measurable 
functions}.

$\Meas$ is well-known to be \emph{complete and cocomplete}: its limits and 
colimits are computed as in $\Set$ and equipped with the initial 
$\sigma$-algebra and final $\sigma$-algebra of the (co)limit diagram, 
respectively.
Any subset $V\subseteq X$ of a measurable space $X$ has a canonical 
$\sigma$-algebra: $\Sigma_V:=\{U\cap V\mid U\in\Sigma_X\}$, called the 
\emph{subspace $\sigma$-algebra}. 
In particular, we can equip set-theoretic \emph{equalizers} with the subspace 
$\sigma$-algebra to get equalizers in $\Meas$. 
We can construct products using the set-theoretic product together with the 
\emph{product $\sigma$-algebra}: the smallest $\sigma$-algebra on 
$\prod_{i\in I}X_i$ generated by $\bigcup_{i\in I}\{\pi_i^{-1}(V)\mid  
V\in\Sigma_{X_i}\}$. 
Note that for countable $I$, this corresponds to the smallest 
$\sigma$-algebra generated by $\{\Pi_{i\in I} V_i\mid \forall_{i\in I} 
V_i\in\Sigma_{X_i}\}$.
The \emph{coequalizer} of $f,g:X\to Y$ is constructed as the set-theoretic 
one $q:Y\to Z$ equipped with the $\sigma$-algebra $\{V\subseteq Z\mid 
q^{-1}(V)\in\Sigma_Y\}$.
We can construct coproducts using the set-theoretic disjoint union together 
with the \emph{coproduct $\sigma$-algebra}: $\{V\subseteq \sum_{i\in 
I}X_i\mid \forall_{i\in I} V\cap X_i\in \Sigma_{X_i}\}$. Note that for 
countable $I$, this is equivalently the collection of sets
$\{\bigcup_{i\in I}\langle i,V_i\rangle \mid \forall_{i\in I} V_i\in\Sigma_{X_i}\}$.
We note that countable coproducts \emph{distribute} over finite products.
 
Crucially, the category of measurable spaces is \emph{not cartesian closed}.
In particular, there is no measurable space structure on the set 
$\Meas(\reals,\reals)$ making the evaluation map
\begin{align*}
\Meas(\reals, \reals)\times \reals & \to \reals\\
\langle f, r\rangle & \mapsto f(r)
\end{align*}
measurable  
\citep{aumann1961borel}, which makes it notoriously challenging to give a 
semantics for higher-order probabilistic programming with continuous 
distributions \citep{heunen2017convenient}.

Every $V\in\Sigma_X$ induces a measurable function $\chi_V:X\to\{0,1\}$, 
called its \emph{characteristic function}.
In fact, the assignment $\Sigma_X\to\Meas(X,[0,\infty])$; $V\mapsto \chi_V$ 
is linear (in the sense that it sends binary disjoint unions to binary sums) 
and Scott continuous (in the sense that it preserves suprema of
$\omega$-chains).
It is an important result in measure theory that $\Meas(X,[0,\infty])$ is 
closed under suprema of $\omega$-chains and $[0,\infty]$-linear combinations
(also called countable $[0,\infty]$-linear combinations) and that every
measurable function $f:X\to [0,\infty]$ is, in fact, the supremum of an
$\omega$-chain of $[0,\infty]$-linear combinations (pointwise convergence)
of characteristic functions \citep[Lemma 2.2.11]{pollard2002user}.
This result is usually known as the approximation by \emph{simple functions}. 

A \emph{measure} $\mu$ on a measurable space $X$ is a linear Scott continuous 
function $\mu:\Sigma_X\to[0,\infty]$ between the $\omega$-cpos. We call 
$U\in\Sigma_X$ such that $\mu(U)=0$ a \emph{($\mu$-)null set}.
We say that a measurable predicate $\chi_A$ on $X$ holds
($\mu$-)\emph{almost everywhere} if $\mu(X\setminus A)=0$.
Because every measurable function is a countable $[0,\infty]$-linear combination
of characteristic functions and $[0,\infty]$ is closed under countable 
$[0,\infty]$-linear combinations, we have a unique linear Scott continuous 
extension along the inclusion of $\Sigma_X$ into the set of measurable 
functions from $X$ to $[0,\infty]$:
\[
\begin{tikzcd}[column sep=large, row sep=large]
\Sigma_X \arrow[r, "\mu"] \arrow[d, hook, "{\chi_{(-)}}"'] & {[0,\infty]} \\
\Meas(X,[0,\infty]) \arrow[ur, "{\quad\int_X \mu(\dif x) (-)(x)}"'] & {}
\end{tikzcd}
\]
called the \emph{integral} with respect to $\mu$ \citep{schilling2017measures}.

We define a \emph{measurable partial function} $X\rightharpoonup Y$ as a 
measurable function $X\to Y+\{\bot\}$. 
A \emph{kernel} from $X$ to $Y$ is a map 
$X\times\Sigma_Y\to[0,\infty]$ that is measurable in its first argument and 
a measure in its second. 
We will sometimes write $k:X\leadsto Y$ to indicate that $k$ is a kernel from
$X$ to $Y$.
We identify measures on $Y$ with kernels from $1$ to $Y$. 
Measurable functions $f:X\times Y \to [0,\infty]$ act on kernels $k:X\times 
\Sigma_Y\to[0,\infty]$ by a canonical action:
$$
{\kact{k}{f}(x,V)} := \int_V k(x,\dif y) \, f(x,y).
$$
The result is still a measure in its second argument because of linearity and 
Scott continuity of integration and scalar multiplication
$[0,\infty]\times[0,\infty]\to[0,\infty]$. 
However, in general, $\kact{k}{f}$ may fail to be measurable in its first
argument, so it is not always strictly a kernel
\footnote{For instance, take $X,Y=\reals$,
$k(x,V\neq \emptyset)=\infty$ and $k(x,\emptyset)=0$
and $f=\chi_W$ for
$W\in \Sigma_{\reals\times\reals}$ such that $\fst(W)\notin \Sigma_\reals$
(whose existence is a classical result in descriptive set theory
\citep{kechris2012classical}, in suitable models of ZF).
Then, $\kact{k}{f}(x,Y)=\int_Y k(x,\dif y) f(x,y)=\chi_{\fst(W)}(x)\cdot
\infty$, which is not measurable by construction of $W$.
}, though, as we shall see later, for well-behaved subclasses
of kernels $k$ (including s-finite kernels),
$\kact{k}{ f}$ is in fact a kernel.

We can \emph{compose kernels} $k:X\times\Sigma_Y\to[0,\infty]$, 
$l:Y\times\Sigma_Z\to[0,\infty]$ by pointwise integration:
$$
(k \dcompose l)(x,V):=\kact{k}{ l(-,V)}(x,Y)=\int_Y k(x,\dif y) \,l(y,V).
$$
For the reasons outlined above, $k\dcompose l$ is a measure in its second
argument.
It also turns out to be measurable in its first argument, as $l(-,V)$ is not
a function of $x$.
Indeed, we can write $l(-,V)$, being a measurable function, as a 
countable linear combination of characteristic functions $\chi_{V_i}$, 
$i\in\naturals$, and use the countable linearity of the integral to obtain that
$$
(k\dcompose l)(x,V)
= \int_Y k(x,\dif 
y)\sum_{i\in\naturals} w_i\chi_{V_i}(y)
=\sum_{i\in\naturals} w_i \int_Y k(x,\dif 
y)\chi_{V_i}(y)
=\sum_{i\in \naturals}w_i k(x, V_i),
$$
i.e.\ a countable linear combination of functions that are measurable in $x$, 
which shows that $k\dcompose l$ is measurable in its first argument and hence 
a kernel.

Similarly, given a kernel $k:X\times\Sigma_Y\to [0,\infty]$ and a measurable 
function $f:Y\to Z$, we can define a \emph{pushforward kernel} 
$f_*k:X\times\Sigma_Z\to[0,\infty]$ by $(f_*k)(x,V):=k(x,f^{-1}(V))$; given a 
measurable function $g:W\to X$, we can define a \emph{pullback kernel} 
{$g^*k:W\times\Sigma_Y\to[0,\infty]$ by $(g^* 
k)(w,V):=k(g(w),V)$.}
We can also \emph{pull back a random variable\footnote{We shall sometimes
use the term random variable as a synonym for measurable function, in particular
if the codomain is (a subset of) $\reals^n$ for some $n\in\naturals$.}
$f:Y\to [0,\infty]$ along a 
kernel} $k:X\times\Sigma_Y\to [0,\infty]$ to give $k^* f: X\to[0,\infty];$ 
$k^*f(x):=\int_Y k(x,\dif y) f(y)$.

For example, we can define the \emph{point mass measure} (or \emph{Dirac 
measure}) $\delta_x$ for every $x\in X$, as $\delta_x(U):=[x\in U]$. 
Furthermore, for $W \in \Sigma_X$, we can construct a \emph{counting measure} on 
$X$ (w.r.t.~$W$), $\counting_W:=\sum_{x\in W}\delta_x = 
\sup_{I\subseteq_{\mathrm{fin}} W}\sum_{i\in I}\delta_i$.
Note that $\counting_W(U)=\infty$ iff $U \cap W$ is infinite.
On $\reals$, we have a unique measure $\lebesgue$, called the \emph{Lebesgue 
measure}, for which $\lebesgue([a,b])=b-a$ \citep{schilling2017measures}.

Given a measure $\mu$ on $X$ and a kernel $\nu$ from $X$ to $Y$, we call a 
measure $\Psi$ on $X\times Y$ a \emph{product measure} of $\mu$ and $\nu$ if 
$\Psi(U\times V)=\int_U\mu(\dif x) \nu(x,V)$, for all
$U\in\Sigma_X$ and $V\in\Sigma_Y$.
By the Carath\'eodory Extension theorem, a maximal such product measure, 
called the \emph{maximal product measure} $\mu\boxtimes \nu$, always exists:
$$
(\mu\boxtimes\nu)(W):=\inf \left\{\sum_{i\in \naturals} \int_{U_i}\mu(\dif x) 
\nu(x,V_i) \mid W\subseteq \bigcup_{i\in\naturals} (U_i\times V_i),\;
U_i\in\Sigma_X,\; V_i\in\Sigma_Y \right\}.
$$
In general, however, there may be many product measures.
For instance, in some cases, it is possible to define product measures 
through iterated integration:
$$
(\mu\otimes^l\nu)(W):=\int_X\mu(\dif x)\int_Y\nu(x,\dif y)\;\chi_W(x,y)
$$
and (in the case where $\nu(x)$ is independent of $x$, i.e.\ $\nu$ is a measure on $Y$)
$$
(\mu\otimes^r\nu)(W):=\int_Y\nu(\dif y)\int_X\mu(\dif x)\;\chi_W(x,y).
$$
In general, $\mu\otimes^l\nu$ might not be well-defined, since 
$x\mapsto\int_Y\nu(x,\dif y)\chi_W(x,y)$ might not be a measurable function 
of $x$
\footnote{As before, take $X,Y=\reals$,
$\nu(x,-):=\infty\cdot\counting_\reals$,
$W\in \Sigma_{\reals\times\reals}$ such that $\fst(W)\notin \Sigma_\reals$
(whose existence is a classical result in descriptive set theory
\citep{kechris2012classical}, in suitable models of ZF).
Then, $x\mapsto\int_Y\nu(x,\dif y)\chi_W(x,y)=\infty\cdot \chi_{\fst(W)}$,
which is not measurable.
}, 
and similarly for $\mu\otimes^r \nu$.
Moreover, even when they are well-defined, they might not be equal:
$(\counting_{\reals}\otimes^l\lebesgue)(\{\langle r,r\rangle\mid r\in 
V\})=0\neq \lebesgue(V)=(\counting_\reals\otimes^r\lebesgue)(\{\langle 
r,r\rangle\mid r\in V\})$ for non-$\lebesgue$-null $V\in\Sigma_\reals$.
That is, \emph{Fubini's theorem} for swapping the order of integration does not 
hold in general.

We note that point mass measures allow us to interpret a measurable function 
$f:X\to Y$ as a kernel $\delta_f$ from $X$ to $Y$.
This lets us relate the pushforward and pullback of kernels to kernel 
integration in the sense that
$$
k\dcompose \delta_f =f_*k \textnormal{\qquad and \qquad} \delta_f\dcompose 
k=f^*k.
$$
Given a measurable partial function $f:X\rightharpoonup Y$, we can define a 
kernel $\delta_f$ from $X$ to $Y$ by setting $\delta_\bot:=0$ (the zero 
measure).
Then, using the two equations above, we can \emph{define} the pushforward and 
pullback of kernels $k$ along $f$. 

We say that a kernel $k$ from $X$ to $Y$ is \emph{supported in $C\in 
\Sigma_{X\times Y}$} if $k(x,Y\setminus C^x)=0$ for all $x\in X$, where we write
$C^x:=\{y\in Y\mid \langle x,
y\rangle\in C\}$ for $C\in\Sigma_{X\times Y}$.

For two kernels $k,l:X\leadsto Y$, let us write $\mutsing{k}{l}$ (read ``$k$ and $l$
are \emph{mutually singular\footnote{
We are using a notion of non-uniform mutual singularity of kernels here.
Sometimes, a more stringent notion is used instead: $k$ and $l$ are called
\emph{uniformly mutually singular} if there exists $A\in\Sigma_Y$ such that
for all $x\in X$, $k(x)(A)=l(x)(Y\setminus A)=0$.
}}'') when there exists $A\in\Sigma_{X\times Y}$ such that $k$ is supported in
$A$ and $l$ is supported in $X\times Y\setminus A$.
Note that any countable family $\{k_n\}_{n\in\naturals}$ is pairwise mutually
singular iff there is a measurable partition $\{A_n\}_{n\in\naturals}$ of
$X\times Y$ such that for all $n\in \naturals$, $k_n$ is supported in $A_n$.
 
For two kernels $k,l: X\leadsto Y$, let us write $k\ll l$ (read ``$k$ is
\emph{absolutely continuous} with respect to $l$'') if for all $x\in X$, for all
$A\in\Sigma_Y$, $l(x)(A)=0$ implies that $k(x)(A)=0$.

\section{Recap: Classes of Kernels}
As we have seen, general measures and kernels can be problematic in the sense
that basic results fail: $\kact{k}{ f}$ might not be measurable,
$\mu\otimes^l k$ might not be well-defined and $\mu\otimes^l \nu$ may not
be equal to $\mu\otimes^r \nu$, even if both sides are well-defined
(Fubini's theorem can fail).
Therefore, we shall now restrict our attention to certain better-behaved classes
of kernels.

In this paper, we shall be interested in the following classes of kernels and 
the corresponding classes of measures, which we identify with kernels with 
domain $1$, a (fixed) singleton set.

\begin{definition}[Classes of Kernels / Measures] 
\label{def:kernel}
{A \emph{kernel} from $(X, \Sigma_X)$ to $(Y, \Sigma_Y)$ is a function 
$k:X\times \Sigma_Y\to[0,\infty]$ such that $k(x, -) : \Sigma_Y \to 
[0,\infty]$ is a measure for all $x \in X$, and $k(-, U) : X \to [0,\infty]$ 
is measurable for all $U \in \Sigma_Y$.
Furthermore, we define the following classes of kernels (and measures, which we 
identify with kernels $\mu:1\times \Sigma_Y\to[0,\infty]$).}
\begin{itemize}

\item $k$ is called a \emph{probability kernel} if 
$k(x,Y)=1$ for all $x\in X$;

\item $k$ is called a \emph{subprobability kernel} if 
{$\sup_{x \in X} k(x, Y) \leq 1$};

\item $k$ is called a \emph{finite kernel} if {$\sup_{x \in X} k(x, Y) < 
\infty$};

\item $k$ is called a (non-uniformly\footnote{There is also a stronger notion of
uniformly $\sigma$-finite kernel $k$ sometimes used in practice, in which case
$k$ is required to decompose into a countable sum of finite kernels that are
\emph{uniformly mutually singular}.
Both reduce to the same usual notion of $\sigma$-finite measure if $X=1$.
In this paper, we shall not be concerned with uniformly $\sigma$-finite kernels.
})
\emph{$\sigma$-finite} if it is of the form $k=\sum_{i\in\naturals}k_i$ where each
$k_i$ is a finite kernel from $(X, \Sigma_X)$ to $(Y, \Sigma_Y)$ and $k_i\bot k_j$
whenever $i\neq j$.	

% \item {$k$ is called a \emph{(uniformly) $\sigma$-finite kernel} if it is of the
% form $k=\sum_{i\in\naturals}k_i$ for finite kernels $k_i$, such that there
% exists a countable (disjoint) partition $Y=\bigcup_{i\in\naturals} Y_i$ where
% each $k_i$ is supported on $Y_i$: $k_i(x,Y_i)=k_i(x,Y)$ for all $x\in X$;}

% {(equivalently, $k$ is a \emph{(uniformly) $\sigma$-finite kernel} if there
% exists a sequence of (pairwise disjoint) sets
% $(Y_i \in \Sigma_Y: i \in \naturals)$ 
% satisfying $\sup_{x \in X} k(x, Y_i) < \infty$ for all $i$, and $\bigcup_i 
% Y_i = Y$)}
\item $k$ is called an \emph{s-finite kernel} if it is of the form 
$k=\sum_{i\in\naturals}k_i$ {where each $k_i$ is a finite kernel from $(X, 
\Sigma_X)$ to $(Y, \Sigma_Y)$}.
\end{itemize}
{It follows from the definition that, ordered by inclusion, the above classes 
form an increasing chain.}
\end{definition}
We stress the uniformity of the bound in the definition of a finite (and 
hence $\sigma$-finite and $s$-finite kernel).

\begin{example}[Deterministic Kernels] Any measurable partial function 
$f:X\rightharpoonup Y$ defines a \emph{deterministic kernel} 
\[
(x, U) \mapsto 
\left\{
\begin{array}{ll}
\delta_{f(x)}(U) & \hbox{if $x \in \dom(f)$}\\
0 & \hbox{otherwise} 
\end{array}
\right.
\]
which is a subprobability kernel and, in fact, a probability kernel if $f$ is 
a function. 
\end{example}

% \begin{example}[Nondeterministic Kernels] Suppose that $k_i$ are disjoint 
% deterministic kernels (in the sense that $k_i(x)= k_j(x)\Rightarrow 
% k_i(x)=\bot$ for all $x\in X, i\neq j\in\naturals$), for $i\in \naturals$.
% Then, $k=\sum_{i\in \naturals} k_i$ is a \emph{(countably) non-deterministic 
% kernel}, a specific example of an s-finite kernel.
% In case for all $x$, there is some finite set $U$ such that $k(x)$ is 
% supported in $U$, we say that $k$ is \emph{finitely non-deterministic}.
% Note that there is a distinct difference between a finitely non-deterministic 
% kernel and a probabilistic kernel supported in finite sets: the latter is 
% normalised while the former is not.
% This distinction is at the root of many of the differences between 
% probability and non-determinism.
% \end{example}

\begin{example}[Non $\sigma$/s-Finite Kernels] An example of an s-finite 
measure that is not $\sigma$-finite is the infinite measure on a one-point space.
An example of a measure that is not s-finite is a counting measure on an
uncountable measurable space like $\reals$.
\end{example}

S-finite kernels have the following important properties.
Their importance lies mainly in the fact that they are a class of 
infinite kernels that is closed under composition (indeed, all the classes 
of kernels in Definition~\ref{def:kernel}, except $\sigma$-finite, are closed 
under composition) and for which most important results from measure theory 
hold.
\begin{theorem}[Composition of s-finite Kernels, 
\citep{staton2017commutative}] \label{thm:sfincomp}
The class of s-finite kernels is closed under
composition (kernel integration): $k \dcompose l$ is s-finite if $k$ and $l$ are.
In particular, it is closed under pushforwards along measurable 
(partial) functions.
\end{theorem}
In particular, it turns out that the more commonly used subclass of 
$\sigma$-finite measures, while generally well-behaved in the sense that, for 
instance, the Fubini and Radon-Nikod\'ym theorems hold, does not enjoy this 
property of compositionality, which explains our preference for s-finite 
kernels.

In fact, it is perhaps for that reason that many texts typically discuss only
$\sigma$-finite measures, rather than $\sigma$-finite kernels, while discussions
of kernels are frequently limited to the (sub)probability case.
Similarly, in this note, we shall focus on measures, rather than kernels,
in the $\sigma$-finite case.

Let us therefore make the definition of $\sigma$-finite measures more explicit: 
a measure $\mu$ on $Y$ is {$\sigma$-finite} if it is the
countable sum $\sum_{i\in \naturals}\mu_i$ of (pairwise) mutually singular
finite measures $\mu_i$.
More explicitly, a measure $\mu$ on $Y$ is
{$\sigma$-finite} if there exists a sequence of (pairwise disjoint) sets 
$(Y_i \in \Sigma_Y: i \in \naturals)$ satisfying $\mu(Y_i) < \infty$ for all 
$i$, and $\bigcup_i Y_i = Y$.

\section{Product Measures and Extension of Measure}
As we shall see, however, s-finite kernels satisfy a limited Fubini theorem
(proved by Staton) in the sense that $\mu\otimes^l\nu=\mu\otimes^r\nu$ if
both sides are defined.
We add the observation that $\mu\otimes^l\nu$ need not equal
$\mu\boxtimes\nu$ (hence the qualification ``limited'').

The classical proof of the Fubini theorem for $\sigma$-finite measures
relies on the uniqueness of the Carath\'eodory extension
\citep{schilling2017measures}.
We note that Carath\'eodory extensions for s-finite measures need not be
unique.

\begin{theorem}[Carath\'eodory Extension Theorem] Let $X$ be a measurable 
space and let $R\subseteq \Sigma_X$ be a sub-ring (or sub-semi-ring) of 
$\Sigma_X$ that generates $\Sigma_X$.
Let $\mu:R\to[0,\infty]$ be a linear Scott continuous map (pre-measure).
Then, there exists an extension $\overline{\mu}:\Sigma_X\to[0,\infty]$ of 
$\mu$.
If $\mu$ is $\sigma$-finite, then there is a unique extension $\overline{\mu}$.
If $\mu$ is s-finite, then there exists an s-finite extension
$\overline{\mu}$, but it may fail to be unique.
In fact, not every extension $\overline{\mu}$ need be s-finite.
\end{theorem}
\begin{proof}
The existence statement and uniqueness for $\sigma$-finite measures are
standard \linebreak \citep[Theorem 6.1]{schilling2017measures}.
The existence of an s-finite extension $\overline{\mu}$ for an s-finite
premeasure $\mu$ follows by noting that $\mu=\sum_{i\in\naturals}\mu_i$,
where $\mu_i$ is finite.
Now, we know that $\mu_i$ has a $\sigma$-finite extension $\overline{\mu_i}$
to all of $\Sigma_X$ and it follows that $\overline{\mu}:=\sum_{i\in\naturals}
\overline{\mu_i}$ is an extension of $\mu$, which is s-finite as a countable
sum of $\sigma$-finite measures (and therefore a countable sum of finite
measures).
The non-uniqueness for s-finite measures follows from the following
counterexample.
Note that the half-open intervals $[a,b)$ form a semi-ring $R$ that generates
the Borel $\sigma$-algebra on $\reals$.
Now, note that $\overline{\mu}=\sum_{n\in\naturals } \uniform{\reals}$ and
$\overline{\mu}+\delta_0$
are two distinct s-finite measures on $\reals$ that restrict to the same
premeasure $\mu$ on $R$.
Moreover, $\counting_{\reals}$ also restricts to $\mu$ on $R$ and it is not
s-finite.
\end{proof}

In fact, we can even show that product measures for s-finite measures are not
unique.
\begin{theorem}[Non-Uniqueness of Product Measures] Product measures for
s-finite measures need not be unique.
In particular, we may have $\mu\otimes^l\nu\neq\mu\boxtimes \nu$ for
s-finite measures $\mu$ and $\nu$.
\end{theorem}
\begin{proof}
Let $\mu=\nu=\infty\cdot\uniform{\reals}$.
Then note that $\mu\otimes^l\nu(\{\langle r,r\rangle\mid r\in\reals\})=
\int_{\reals\times\reals}\mu\otimes^l\nu(\dif z)
[\fst(z)=\snd(z)]=\infty\cdot \int_\reals \uniform{\reals}(\dif x)\int_\reals
\uniform{\reals}(\dif y)[x=y]=\infty\cdot 0=0$.
Meanwhile, for the maximal product measure
$\mu \boxtimes \nu(\{\langle r,r\rangle\mid r\in\reals\})=
\inf\{\sum_{i\in\naturals}\mu(A_i)\nu(B_i)\mid
\{\langle r,r\rangle\mid r\in\reals\}\subseteq
\bigcup_{i\in\naturals}A_i\times B_i,\; A_i,B_i\in\Sigma_\reals\}
=\infty$.
Indeed, in any such cover we have
$\reals=\bigcup_i(A_i\cap B_i)$; hence some $A_i\cap B_i$ has positive
Lebesgue measure, and for this $i$ both $\mu(A_i)$ and $\nu(B_i)$ are infinite.
\end{proof}
This shows that the classical integral recipe
$$
\int_{X\times Y} = \int_X \int_Y
$$
can fail for s-finite measures.
(Often, one regards the maximal product measure $\mu\boxtimes \nu$ as the
canonical choice of product measure.)

However, we do have the following limited Fubini theorem for
s-finite kernels.
Let us define parameterised versions of $\otimes^l$ and $\otimes^r$:
for $k:X\leadsto Y$ and $l:X\times Y\leadsto Z$ write
$$
(k\otimes^l l)(x,W) = \int_Y k(x,\dif y) \int_Z l(x,y,\dif z) \chi_W(y,z)
$$
and for
$k:X\leadsto Y$ and $l:X\leadsto Z$, write
$$
(k\otimes^r l)(x,W) = \int_Z l(x,\dif z) \int_Y k(x,\dif y) \chi_W(y,z).
$$
Then, we have the following.

\begin{theorem}[Limited Fubini \citep{staton2017commutative}] 
\label{thm:fubini} If $k:X\leadsto Y$ and $l:X\times Y\leadsto Z$ are s-finite,
then $k\otimes^l l$ defines an s-finite kernel $X\leadsto Y\times Z$. 
Furthermore, if $l(x,y,-)$ does not depend on $y\in Y$ (so that we
may regard $l$ as a kernel $X\leadsto Z$), then $k\otimes^r l$ also
defines an s-finite kernel $X\leadsto Y\times Z$ and
$$k\otimes^l l= k \otimes^r l.$$
\end{theorem}
This shows that the classical integral recipe
$$
\int_X \int_Y=\int_Y\int_X
$$
is valid for s-finite measures.

\section{Recap: Standard Borel Spaces}
General measure spaces are too wild for many desirable results to hold, but 
usually we are only interested in a very well-behaved subclass of them, the 
\emph{standard Borel spaces}, which has several useful characterisations as 
follows (see e.g.~Kuratowski's Classification Theorem \citep[Section 
15B]{kechris2012classical}). 

\begin{definition}[Standard Borel Space]
We call a measurable space $X$ a \emph{standard Borel space} if its underlying
set admits a complete metric with a dense countable subset whose Borel
$\sigma$-algebra is $\Sigma_X$.
\end{definition}

\begin{proposition}[Kuratowski's Classification Theorem]
Given a measurable space $(X, \Sigma_X)$ the following are equivalent:
\begin{enumerate}
\item $X$ is a standard Borel space.
\item $(X, \Sigma_X)$ is either measurably isomorphic to $(\reals, 
\Sigma_\reals)$ with the Borel $\sigma$-algebra $\Sigma_\reals$, or countable 
with the discrete $\sigma$-algebra.
\item $(X, \Sigma_X)$ is a measurable retract of $(\reals, \Sigma_\reals)$, 
i.e., there exist measurable
$X \xrightarrow{f} \reals \xrightarrow{g} X$ such that $g \circ f = 
\mathrm{id}_X$.
\end{enumerate}
\end{proposition}

Note that all singletons are measurable in a standard Borel space.
We write $\Sbs$ for the full subcategory of $\Meas$ on the standard Borel 
spaces.

We record some non-trivial results for standard Borel spaces here. 
The subcategory $\Sbs\subseteq \Meas$ is closed under countable products, 
countable (distributive) coproducts and measurable subspaces (in particular 
equalizers) \citep[Section 12B]{kechris2012classical}. 
A function $f:X\to Y$ between standard Borel spaces is measurable iff its 
graph is measurable \citep[Section 14C]{kechris2012classical}. 
If $f:X\to Y$ is a measurable injection between standard Borel spaces, then 
$f(U)$ is measurable for any measurable $U\in\Sigma_X$, i.e.\ a measurable 
injection between standard Borel spaces is an embedding \citep[Section 
15A]{kechris2012classical}. 
In particular, a measurable bijection between standard Borel spaces is an 
isomorphism.

\section{Characterising S-finite Kernels}
A downside of s-finite kernels is that they have hardly been studied by 
probabilists, so even basic results still need to be established for them.
For this purpose, we give a few characterisations of s-finite 
measures and kernels.

\begin{theorem}[Characterising s-finite Kernels] \label{thm:charac-sfinite}
We have the following equivalent characterisations of the s-finiteness of a 
kernel $\nu$ from $X$ to $Y$:
\begin{enumerate}
\item $\nu = \sum_{n\in \naturals} \nu_n$ for subprobability kernels $\nu_n$;
\item $\nu$ is the pushforward of a $\sigma$-finite kernel.\\

Moreover:
\item for any measurable $f:X\times Y\to [0,\infty]$, $\kact{\nu}{ f}$ 
is an s-finite kernel if $\nu:X\times\Sigma_Y\to[0,\infty]$ is; in 
particular, $\kact{\nu}{ f}$ is s-finite if $\nu$ is a subprobability 
kernel;
\item (a weak converse:) if $\nu:X\times\Sigma_Y\to[0,\infty]$ is an
s-finite kernel, then there
exists a subprobability kernel $\mu:X\times \Sigma_Y\to[0,1]$ and a function
$f: X\times Y\to [0,\infty]$ (with $f(x,-)$ measurable for all $x\in X$)
such that $\nu = \kact{\mu }{ f}$ and $\mu(x)([f(x,-)=0])=0$ for all
$x\in X$.
(So we may choose $f>0$.)
If either $X$ is countable and discrete or $Y$ is standard Borel,
$f$ can be taken to be jointly measurable in $X$ and $Y$.
\end{enumerate}
\end{theorem}
\begin{proof}
For 1), first decompose $\nu$ as the sum $\sum_{n\in\naturals}\nu_n$ of 
finite kernels.
Write $I_n$ for the smallest integer larger than $\sup_{x\in X}\nu_n(x,Y)$.
Then, note that $\nu = \sum_{n\in\naturals}\sum_{1\leq i\leq I_n} \nu_n/I_n$, 
where $\nu_n/I_n$ is a subprobability kernel and $\sum_{n\in\naturals}\{1\leq 
i\leq I_n\}\cong \naturals$.
Conversely, characterisation 1) is clearly a special case of our definition.

2) is a generalisation of Proposition 7 of \citep{staton2017commutative} from 
s-finite measures to s-finite kernels.
The proof is virtually the same.
It is clear that the pushforward $\nu$ of a $\sigma$-finite kernel $\mu$ is 
s-finite, since $\sigma$-finite kernels are, in particular, s-finite 
and s-finite kernels are closed under composition.
Conversely, given an s-finite kernel $\nu$ from $X$ to $Y$, we define a 
$\sigma$-finite kernel $\mu$ from $X$ to $\naturals\times Y$.
Indeed, decompose $\nu$ as a sum of finite kernels $\sum_{i\in\naturals}\nu_i$.
Then, define $\mu(x,V):=\sum_{i\in \naturals}\nu_i(x, \snd(V\cap (\{i\}\times 
Y)))$.
Note that $\mu$ is $\sigma$-finite and that $\nu=\snd_*\mu$.
(In fact, $\mu$ can even be observed to be a uniformly $\sigma$-finite kernel.)

For 3), note that we can approximate $f$ as a countable sum of bounded 
functions $\{f_n\}_{n\in \naturals}$ (the usual approximation by simple 
functions).
Let us decompose $\nu$ as a sum $\sum_{m\in\naturals}\nu_m$ of finite 
kernels. Then, $\kact{\nu}{ f}= \kact{(\sum_{m\in\naturals} 
\nu_m)}{(\sum_{n\in\naturals} f_n)}=\sum_{\langle n,m\rangle \in
\naturals\times\naturals} \kact{ \nu_m}{ f_n}$, which is a countable sum 
of finite kernels.
Indeed, note that $\kact{\nu_m}{ f_n}(x,Y)\leq (\sup_{(x,y)\in X\times Y} 
f_n(x,y))\cdot (\sup_{x\in X} \nu_m(x,Y))<\infty$.
(It is a classical result in measure theory that
$\kact{\nu_m}{f_n}$ defines a finite kernel (in particular, is measurable)
for $\nu_m$ a finite kernel and $f_n$ a bounded measurable function
\citep[Theorem 4.20 (ii)]{pollard2002user}).
For 4), decompose $\nu$ as the sum $\sum_{n\in\naturals}\nu_n$ of
subprobability kernels. 
Then define $\mu := \sum_{n\in\naturals} \nu_n/2^{n+1}$. 
It then follows that $\mu$ is a subprobability kernel. 
Note that, for all $x\in X$, $\nu_n(x)\ll\nu(x)\ll \mu(x)$ and define
$f_n(x,-) := \dif \nu_n(x)/\dif \mu(x)$ and $f:= 
\sum_{n\in\naturals}f_n$ (using the Radon-Nikod\'ym theorem for finite 
measures \citep{kallenberg2006foundations}). 
It then follows that
\begin{align*}
    \kact{\mu }{ f}(x,A)
&= \int_A \mu(x,\dif y)\; f(x,y)  \\
&= \int_A \mu(x,\dif y)\;\sum_{n\in\naturals}f_n(x,y) \\
&=\sum_{n\in\naturals}\int_A \mu(x,\dif y)\; f_n(x,y) \\
&=\sum_{n\in\naturals}\int_A \mu(x,\dif y)\;\dif\nu_n(x)/\dif\mu(x)(y) \\
&=\sum_{n\in\naturals}\nu_n(x,A)
=\nu(x,A),
\end{align*} where we use the monotone convergence theorem
\citep{schilling2017measures} 
to pull the countable sum out of the integral.
Note that
\begin{align*}
\nu(x,[f=0]) &= \int_{[f=0]} \; \nu(x,\dif y)\\
&= \int_{[f=0]} \mu(x,\dif y)\; f(x,y) \\
&= \int_{[f=0]} \mu(x,\dif y)\;0\\
&= 0.
\end{align*}
Now, as $\mu(x)\ll\nu(x)$, it follows that also $\mu(x,[f=0])=0$.
Finally, we note that if $X$ is countable and discrete or
$Y$ is standard Borel,
we can apply the Radon-Nikod\'ym theorem for
finite kernels (\citep[Theorem 1.28]{kallenberg2017random}), which lets us
construct $f_n$ (hence $f$) uniformly for all $x\in X$ as the kernel
RN-derivative $\dif \nu_n/\dif\mu$ and hence $f$ is jointly measurable
in this case.
\end{proof}
In particular, we see that, for $X$ countable and
discrete or $Y$ standard Borel,
a kernel $\nu$ from $X$ to $Y$ is s-finite if and only if it is of the shape
$\kact{\mu}{ f}$ for $\mu$ a subprobability kernel and
$f:X\times Y\to[0,\infty]$ measurable.

Compare points 3 and 4 with the following characterisation of $\sigma$-finite
kernels.
\begin{theorem}[Characterising $\sigma$-finite Kernels] 
\label{thm:characsigmaf}
Let $\nu$ be a kernel from $X$ to $Y$.
Then,
\begin{enumerate}
\item for any measurable function $f:X\times Y\to [0,\infty)$,
$\kact{\nu}{ f}$ is a $\sigma$-finite kernel if $\nu$ is; in particular,
$\kact{\nu}{ f}$ is $\sigma$-finite if $\nu$ is a subprobability kernel; 
\item conversely, if $\nu:X\times\Sigma_Y\to[0,\infty]$ is a
$\sigma$-finite kernel, then there
exists a subprobability kernel $\mu:X\times \Sigma_Y\to[0,1]$ and a measurable
function $f: X\times Y\to (0,\infty)$ such that $\nu=\kact{\mu}{ f}$;
\item $\nu$ is $\sigma$-finite iff there exists a measurable
function $g:X\times Y\to (0,\infty)$ such that
$\kact{\nu}{ g}$ is a subprobability (equivalently, finite) kernel.
\end{enumerate}
\end{theorem}
\begin{proof}
1. Let us write $B_m:= f^{-1}([m,m+1))$ and $f^m$ for the measurable function
that is equal to $f$ on $B_m$ and $0$ elsewhere.
Note that $f=\sum_{m\in\naturals} f^m$.
Note further that $f^m$ is bounded by $m+1$.

Decompose $\nu$ as a sum $\sum_{n\in\naturals}\nu_n$ of mutually singular
finite kernels such that
\begin{itemize}
\item $A_n\in\Sigma_{X\times Y}$;
\item the sets $A_n$ partition $X\times Y$;
\item for all $x\in X$, $n\neq n'$ implies that $\nu_n(x,A_{n'}^x)=0$.
\end{itemize}

Then, $\kact{\nu}{ f}=
\kact{(\sum_{n\in\naturals}\nu_n)}{(\sum_{m\in\naturals}f^m)}
= \sum_{n,m\in\naturals} \kact{\nu_n}{ f^m}$.

We are done if we can show that the $\kact{\nu_n}{ f^m}$ are mutually
singular.
To observe this, note that
\begin{itemize}
\item $A_n\cap B_m\in \Sigma_{X\times Y}$;
\item the sets $A_n\cap B_m$, $(n,m\in\naturals)$, partition $X\times Y$ as
the sets $A_n$, $(n\in\naturals)$, and $B_m$, $(m\in\naturals)$ partition
$X\times Y$; 
\item $
\kact{\nu_n}{ f^m}(x)$ is clearly supported in
$(A_n\cap B_m)^x= A_n^x\cap B_m^x$.
\end{itemize}

We conclude that $\kact{\nu}{ f}$ is a $\sigma$-finite kernel.

2.
Decompose $\nu=\sum_{n\in\naturals}\nu_n$ as a sum of mutually singular finite
kernels.
Let $\{A_n\}_{n\in\naturals}$ be a measurable partition of $X\times Y$ such
that $\nu_n$ is supported in $A_n$.
Let $g(x,y):=(2^{n+1}\cdot \max\{1,\sup_{x\in X}\nu_n(x, Y)\})^{-1}$ for
$(x,y)\in A_n$.
Note that $g:X\times Y\to (0,\infty)$ is measurable.
Then,
$\sup_{x\in X}(\kact{\nu}{ g})(x,Y)\leq \sum_{n\in \naturals} 1/2^{n+1}=1$,
so $\mu=\kact{\nu}{ g}$ is a subprobability kernel.
Note that $f:=1/g:X\times Y\to (0,\infty)$ is also measurable and that
$\nu=\kact{\mu}{ f}$.

3. Suppose that $\nu$ is $\sigma$-finite.
The function $g$ in the proof of 2. has the required property.

Conversely, suppose there is a $g:X\times Y\to (0,\infty)$ such that
$\mu=\kact{\nu}{ g}$ is a subprobability kernel.
Note that $f=1/g:X\times Y\to (0,\infty)$ is measurable and that
$\nu=\kact{\mu}{ f}$.
Then, by the first part, it follows that $\nu$ is $\sigma$-finite.
\end{proof}

\begin{theorem}[Another Characterisation of S-finite Kernels]\label{thm:sfinalt}
Assume either that $X$ is countable and discrete or that $Y$ is a standard Borel
space. A kernel 
$\nu:X\leadsto Y$ is s-finite iff there exists a $\sigma$-finite kernel
$\mu:X\leadsto Y$ and a jointly measurable function
$f:X\times Y\to \{1,\infty\}$ such that 
$\nu=\kact{\mu}{ f}$.
Moreover, for all $x\in X$, $f(x,-)$ is unique $\nu(x)$-almost everywhere
and, for $A\in\Sigma_Y$, $\mu(x,A)$ is unique if $\infty\notin f(x,A)$.
\end{theorem}
\begin{proof}
Suppose that $\nu$ is s-finite.
By part 4 of Theorem~\ref{thm:charac-sfinite} and the present hypothesis, we get
a subprobability kernel $\mu'$ and a jointly measurable function
$f':X\times Y\to[0,\infty]$ such that $\nu=\kact{\mu'}{ f'}$.
Let $f(x,y):=f'(x,y)$ if $f'(x,y)=\infty$ and $f(x,y)=1$ otherwise.
Then $f:X\times Y\to\{1,\infty\}$ is jointly measurable.
Let $g(x,y):=f'(x,y)$ if $f'(x,y)\neq \infty$ and $g(x,y)=1$ otherwise.
Then $g:X\times Y\to [0,\infty)$ is jointly measurable.
Define $\mu:= \kact{\mu'}{ g}$ and observe that it is $\sigma$-finite by
part 1 of Theorem~\ref{thm:characsigmaf}.
Moreover, as $f'=g\cdot f$, we have
$\nu=\kact{\mu'}{ f'}= \kact{\mu'}{ (g\cdot f)}=
\kact{(\kact{\mu'}{ g})}{ f}=\kact{\mu}{ f}$.

Conversely, suppose that $\nu=\kact{\mu}{ f}$ for $\sigma$-finite $\mu$ and
jointly measurable $f:X\times Y\to \{1,\infty\}$.
Note that $\mu$ is in particular s-finite.
Then, by part 3 of Theorem~\ref{thm:charac-sfinite}, it follows that $\nu$ is s-finite.

For the uniqueness statement, suppose also that $\nu=\kact{\mu'}{ f'}$ with
$\mu'$ $\sigma$-finite and $f':X\times Y\to\{1,\infty\}$ jointly measurable.
That is, for all $x\in X$ and $A\in \Sigma_Y$, we have
$$
\int_A \mu(x,\dif y) f(x,y)=\nu(x,A)=\int_A \mu'(x,\dif y) f'(x,y).
$$
It is clear that $f$ and $f'$ can differ on a $\nu(x,-)$-null set.
Suppose, for instance, that $f=1$ but $f'=\infty$ on some measurable
$\nu(x,-)$-non-null set $A$.
Since both $\mu(x,-)$ and $\mu'(x,-)$ are $\sigma$-finite, there exists a
measurable subset $B\subseteq A$ such that $\mu(x,B)<\infty$ and
$0<\mu'(x,B)<\infty$.
It then follows that
$$
\nu(x,B)=\int_B \mu'(x,\dif y) f'(x,y)=\int_B\mu'(x,\dif y) \infty=\infty
$$
while also
$$
\nu(x,B)=\int_B \mu(x,\dif y) f(x,y)=\int_B\mu(x,\dif y)=\mu(x,B)<\infty,
$$
which is a contradiction.
The case where $f=\infty$ but $f'=1$ on a $\nu(x,-)$-non-null set is symmetric.
It follows that $f(x,-)$ and $f'(x,-)$ only differ on a $\nu(x,-)$-null set.

Suppose that $\infty\notin f(x,A)$.
Then, it follows that 
$$
\mu(x,A)=\int_A \mu(x,\dif y)=\int_A\mu'(x,\dif y) =\mu'(x,A).
$$
\end{proof}

To establish Radon-Nikod\'ym and disintegration theorems for 
s-finite kernels later, we shall need the following definition of what we 
shall call a $0$-$\infty$-set to complement that of a null set.
Let us say that $U\in \Sigma_X$ is an \emph{$0$-$\infty$-set} with respect to a measure
$\mu$ on $X$ if for all $V\in\Sigma_U$ we have $\mu(V)=0$ or $\infty$.
In particular, any $\mu$-null set is a $\mu$-$0$-$\infty$-set, which we shall refer
to as a \emph{trivial $0$-$\infty$-set}.
We note that $\sigma$-finite measures $\mu$ do not have any non-trivial
$0$-$\infty$-sets because {any set of infinite $\mu$-measure must have a countable
partition of finite $\mu$-measure}.
As we shall see in Theorem~\ref{thm:topinftysets}, the possession of non-trivial
$0$-$\infty$-sets is a key distinguishing feature of s-finite measures compared to
$\sigma$-finite measures.

In some cases (in particular, if $\mu$ is s-finite), it turns out that there is,
in some sense, a largest $0$-$\infty$-set $\infty[\mu]$.
Indeed, observe that we can always obtain
another $\mu$-$0$-$\infty$-set from a given one by taking its union with some
$\mu$-null set.
We call a $\mu$-$0$-$\infty$-set $A$ a \emph{top $0$-$\infty$-set} if for all other
$0$-$\infty$-sets $B$, we have $\mu(B\setminus A)=0$.
It is clear that, if such a top $0$-$\infty$-set exists, it is unique up to
null sets and we shall write $\infty[\mu]$ for it.

\begin{theorem}\label{thm:topinftysets}
An s-finite measure $\mu$ on $X$ has a ($\mu$-a.e.\ unique) top $0$-$\infty$-set
$\infty[\mu]$.
$\mu$ is $\sigma$-finite iff $\infty[\mu]$ is trivial.
\end{theorem}
\begin{proof}
For the second statement, note that, since a $\sigma$-finite measure $\mu$ arises as
a sum of mutually singular finite measures, we have $\mu(A)=\infty$ implies
that there is some $B\in \Sigma_A$ such that $0< \mu(B)<\infty$.
This shows that $\sigma$-finite measures only have trivial $0$-$\infty$-sets.
It follows that $\infty[\mu]$ is trivial for a $\sigma$-finite measure.

Now, suppose that $\mu$ is a more general s-finite measure.
By Theorem~\ref{thm:sfinalt}, we obtain a $\sigma$-finite measure $\nu$ on $X$
together with a measurable function $f:X\to \{1,\infty\}$ such that
$\mu=\kact{\nu}{ f}$.
The claim is that $f^{-1}(\infty)$ is a top $0$-$\infty$-set for $\mu$.
Indeed, observe that $\mu|_{X\setminus f^{-1}(\infty)}=
\kact{\nu|_{X\setminus f^{-1}(\infty)}}{ f|_{X\setminus f^{-1}(\infty)}}$
is $\sigma$-finite by part 1 of Theorem~\ref{thm:characsigmaf}, since $\nu$ is
$\sigma$-finite and $f|_{X\setminus f^{-1}(\infty)}<\infty$.
Therefore, $\mu|_{X\setminus f^{-1}(\infty)}$ does not have any non-trivial
$0$-$\infty$-sets.
It follows that $f^{-1}(\infty)$ is a top $0$-$\infty$-set.

Finally, suppose that $\infty[\mu]$ is trivial.
Since top $0$-$\infty$-sets are a.e.\ unique, this means that
$\mu(f^{-1}(\infty))=0$ and therefore also $\nu(f^{-1}(\infty))=0$.
Let $g(x)=f(x)$ if $f(x)<\infty$ and $g(x)=1$ otherwise.
Then, it follows that $\mu=\kact{\nu}{ f}=\kact{\nu}{ g}$.
However, $\nu$ is $\sigma$-finite and $g<\infty$, so part 1 of Theorem~\ref{thm:characsigmaf} implies that $\mu$ is $\sigma$-finite.
\end{proof}

For s-finite measures, non-trivial $0$-$\infty$-sets are, in a sense, sets of
infinite measure ``in a bad way''.
\begin{lemma}
For every s-finite measure $\mu$ on $X$, a measurable set $U$ is not finitely 
approximable in the sense that
$$
\mu(U)\neq \sup\{\mu(V) \mid V\in\Sigma_U,\ \mu(V)<\infty\}
$$
iff $U\cap \infty[\mu]$ is a non-trivial $0$-$\infty$-set and 
$\mu(U\setminus \infty[\mu])<\infty$.
In particular, every non-trivial $0$-$\infty$-set is not finitely approximable.
\end{lemma}
\begin{proof}
Let $I=\infty[\mu]$.
By Theorem~\ref{thm:topinftysets}, $\mu$ is $\sigma$-finite on $X\setminus I$,
so every measurable subset of $X\setminus I$ is finitely approximable.
If $\mu(U\cap I)=0$, then finite approximations inside $U\setminus I$ already
approximate $\mu(U)$.
If $\mu(U\cap I)=\infty$, then every finite-measure subset of $U$ has
$\mu$-null intersection with $I$, and the supremum of the finite measures of
subsets of $U$ is therefore $\mu(U\setminus I)$ if this is finite, and is
$\infty$ otherwise.
The displayed equivalence follows.
\end{proof}

Our conclusion is that every s-finite measure $\mu$ on $X$ decomposes into a
``good'' $\sigma$-finite part on $X\setminus \infty[\mu]$ and a
``badly infinite'' part on $\infty[\mu]$.

We briefly make an observation about the situation for s-finite kernels.
\begin{lemma}\label{lem:kernelinftyset}
Let $k:X\leadsto Y$ be an s-finite kernel with $X$ countable and discrete or
$Y$ standard Borel.
Then, there exists a set $\infty[k]\in\Sigma_{X\times Y}$ such that
$\infty[k]^x$ is a top $0$-$\infty$-set for $k(x)$ for all $x\in X$.
Moreover, $\kact{k}{ \chi_{X\times Y\setminus \infty[k]}}$ is
$\sigma$-finite.
\end{lemma}
\begin{proof}
By Theorem~\ref{thm:sfinalt}, we have $k=\kact{l}{ f}$ for a
$\sigma$-finite kernel $l:X\leadsto Y$ and $f:X\times Y\to \{1,\infty\}$
measurable.
We can define $\infty[k]:=f^{-1}(\infty)$, which is measurable.
Then, it is clear from the preceding argument that $\infty[k]^x$ is a top $0$-$\infty$-set
for $k(x)$ for all $x\in X$.
Finally, observe that
$\kact{k}{\chi_{X\times Y\setminus \infty[k]}}=
\kact{l}{\chi_{X\times Y\setminus \infty[k]}}$ to see that 
$\kact{k}{\chi_{X\times Y\setminus \infty[k]}}$ is $\sigma$-finite by
part 1 of Theorem~\ref{thm:characsigmaf}.
\end{proof}

We call such a set $\infty[k]$ a top $0$-$\infty$-set for the kernel $k$.

\section{Radon-Nikod\'ym}
In practice, rather than describing probabilistic models in terms of measures
and kernels, one often simply describes their \emph{density} with respect to
some reference measure (usually the Lebesgue measure or a counting measure).
Let us therefore turn to the question of when such densities exist for
s-finite kernels.
The answer is given by the Radon-Nikod\'ym Theorem.

\begin{definition}[Radon-Nikod\'ym Derivative/Density] Let $\nu$ and $\nu'$ 
be kernels from $X\leadsto Y$.
By a \emph{Radon-Nikod\'ym derivative} (or density) of $\nu'$ with respect to 
$\nu$, we mean a function $f:X\times Y\to [0,\infty]$, with $f(x,-)$ measurable
for every $x\in X$, such that $\nu'= \kact{\nu}{ f}$.
We sometimes write $\dif \nu'/\dif \nu$ for $f$.
\end{definition}
Note that, in general, a density $f$ of a kernel $\nu':X\leadsto Y$ with respect
to $\nu:X\leadsto Y$ is simply a collection $\{f(x,-)\}_{x\in X}$ of densities
of $\nu'(x)$ with respect to $\nu(x)$.
In practice, however, we shall be particularly interested in cases where $f$
can be taken to be jointly measurable in $X$ and $Y$.

Recall that $\nu'$ is \emph{absolutely continuous} with respect to $\nu$, 
written $\nu'\ll \nu$, if, for all $x\in X$, for all $\nu(x)$-null sets $U$,
$U$ is a $\nu'(x)$-null set.
For kernels $\nu',\nu:X\leadsto Y$, let us say that $\nu'$ is 
\emph{$0$-$\infty$-absolutely continuous} (write $\nu'\lli \nu$) with respect to 
$\nu$ if $\nu'\ll\nu$, and for all $x\in X$, for all $\nu(x)$-$0$-$\infty$-sets $U$,
$U$ is a $\nu'(x)$-$0$-$\infty$-set.
Note that for a $\sigma$-finite kernel $\nu$, we have $\nu'\lli\nu$ iff 
$\nu'\ll \nu$, vacuously.

\begin{lemma}
Suppose that $\nu,\nu':X\leadsto Y$ are kernels such that for all $x\in X$,
$\nu(x),\nu'(x)$ have top $0$-$\infty$-sets
(for instance, if $\nu,\nu'$ are s-finite).
Then, $\nu'\lli \nu$ iff $\nu'\ll \nu$ and for all $x\in X$
$$
\nu'(x,\infty[\nu(x)]\setminus\infty[\nu'(x)])=0.
$$
\end{lemma}
\begin{proof}
Suppose that $\nu'\lli \nu$.
Then, by definition $\nu'\ll \nu$.
Moreover, $\infty[\nu(x)]$ is a $\nu(x)$-$0$-$\infty$-set, so it is also a
$\nu'(x)$-$0$-$\infty$-set and is therefore $\nu'(x)$-almost everywhere contained
in $\infty[\nu'(x)]$.

Conversely, suppose that $\nu'\ll \nu$ and for all $x\in X$
$$
\nu'(x,\infty[\nu(x)]\setminus\infty[\nu'(x)])=0.
$$
Let $U\in \Sigma_Y$ be a $\nu(x)$-$0$-$\infty$-set.
Then, $\nu(x,U\setminus \infty[\nu(x)])=0$ as $\infty[\nu(x)]$ contains all
$\nu(x)$-$0$-$\infty$-sets $\nu(x)$-almost everywhere.
Therefore, $\nu'(x,U\setminus\infty[\nu(x)])=0$ as $\nu'(x)\ll\nu(x)$.
It follows from the displayed assumption that
$\nu'(x,U\setminus \infty[\nu'(x)])=0$.
Thus $U$ is $\nu'(x)$-almost everywhere contained in $\infty[\nu'(x)]$, so $U$
is a $\nu'(x)$-$0$-$\infty$-set.
\end{proof}

Next, we illustrate in what sense this concept is relevant to the question of
the existence of densities.
\begin{lemma}\label{lem:tri-abs-cts}
Suppose that $\nu'$ has a density $f:X\times Y\to [0,\infty]$ with respect to
$\nu$, i.e.~$\nu'=\kact{\nu}{ f}$. Then $\nu'\lli \nu$.
\end{lemma}
\begin{proof}
Note that $\nu(x,U)=0$ implies that $\nu'(x,U)=\int_U \nu(x,\dif y) f(x,y)=0$. 

Moreover, suppose that $U$ is a $\nu(x)$-$0$-$\infty$-set.
Let $V\in\Sigma_U$.
Then 
\[
\nu'(x,V)=\int_V\nu(x,\dif y)f(x,y)
=\sum_{i\in\naturals} w_i \nu(x,V_i)=0\textnormal{ 
or } \infty
\] 
where we decompose $f(x,-)$ as a countable sum $\sum_{i\in\naturals}w_i 
\chi_{V_i}$ of characteristic functions of measurable subsets $V_i$ of $V$.
As $V_i\in \Sigma_U$, we have $\nu(x,V_i)=0\textnormal{ or }\infty$.
\end{proof}

It is easy to see that such Radon-Nikod\'ym derivatives have a uniqueness 
property, if they exist; this justifies the notation $\dif \nu'/\dif \nu$.
For this purpose, let us say that two kernels $k,l$ from $X$ to $Y$ are 
\emph{almost everywhere $\infty$-equal} with respect to some measure 
$\mu$ on $X$ if
$$
\kact{k}{ (1+\infty\cdot \chi_{\infty[\mu]})}\textnormal{\qquad and\qquad}
\kact{l}{ (1+\infty\cdot \chi_{\infty[\mu]})}
$$
are $\mu$-almost everywhere equal.
That is, $k$ and $l$ are $\mu$-almost everywhere equal on
$X\setminus \infty[\mu]$ and $ \kact{k}{ \infty}$ and
$\kact{l}{ \infty}$ are
$\mu$-almost everywhere equal on $\infty[\mu]$.

Similarly, we say that a kernel $k$ from $X$ to $Y$ is \emph{almost 
everywhere $\infty$-unique} with respect to some predicate $P$ on kernels if 
all kernels that satisfy $P$ are almost everywhere $\infty$-equal.

\begin{theorem}[Uniqueness of RN-Derivatives] \label{thm:uniquern}
Let $\nu$ and $\nu'$ be s-finite measures on $X$.
Suppose $\nu'=\kact{\nu}{ f}$, and let $g:X\to [0,\infty]$ be measurable.
Then, we have 
\[
\nu'=\kact{\nu}{ g} \; \iff \; \nu([f\neq g]\cap (X\setminus 
\infty[\nu]))+\nu([g = 0 \neq f]\cap \infty[\nu])
+\nu([f = 0 \neq g]\cap \infty[\nu])=0.
\]
That is, $f$ and $g$ are almost everywhere $\infty$-equal: RN-derivatives are 
almost everywhere $\infty$-unique.
In other words, on the $\nu$-$\sigma$-finite part of $X$, $f$ and $g$ are
$\nu$-a.e.~equal; on 
its complement, the points where one has value 0 and the other has a strictly 
positive value are $\nu$-negligible.
\end{theorem}
\begin{proof}
Note that $\nu$ is $\sigma$-finite when restricted to
$X\setminus \infty[\nu]$.
Here, $\nu([f\neq g]\cap (X\setminus \infty[\nu]))=0$ is well-known to be the 
precise uniqueness property of Radon-Nikod\'ym derivatives (see e.g. 
\citep{pollard2002user}).

So let us restrict our attention to $U\in \Sigma_{\infty[\nu]}$.
Then, either $\nu(U)=0$ or $\nu(U)=\infty$.
If $\nu(U)=0$, then it automatically follows that $\kact{\nu}{ 
f}(U)=0=\kact{\nu}{ g}(U)$.
If $\nu(U)=\infty$, then we must have $\kact{\nu}{ f}(U)=0$ iff 
$\kact{\nu}{ g}(U)=0$, i.e., 
$\nu([g = 0 \neq f]\cap \infty[\nu])+\nu([f = 0 \neq g]\cap \infty[\nu])=0$.
\end{proof}
Note that this uniqueness theorem applies, in particular, to s-finite kernels
as s-finite kernels are pointwise s-finite measures and densities of kernels
are simply pointwise densities of measures.

Showing the existence of such Radon-Nikod\'ym derivatives is less 
straightforward, but it is well-known that this can be done for 
$\sigma$-finite measures
(see, e.g., \citep[Theorem 5.5.4]{Dudley04} and \citep[Theorem 4.2.3]{Cohn80}).
We generalise this result to the s-finite setting.

\begin{theorem}[Radon-Nikod\'{y}m Theorem for S-finite Kernels]
\label{thm:radon-nikodym}
Let $\nu,\nu':X\leadsto Y$ be s-finite kernels such that 
$\nu'\lli\nu$. 
Then, there exists a function $\dif \nu'/\dif \nu:X\times Y\to[0,\infty]$, 
called a Radon-Nikod\'ym derivative of $\nu'$ with respect to $\nu$, for which 
$\dif\nu'/\dif \nu(x,-)$ is measurable in $Y$ for all $x\in X$, and
$\nu'= \kact{\nu}{ \dif\nu'/\dif \nu}$.
If $X$ is countable and discrete or
$Y$ is standard Borel, $\dif\nu'/\dif \nu$ can be
taken to be jointly measurable in $X$ and $Y$.

We note that this theorem can fail if we only require $\nu'\ll\nu$, as 
would be customary for $\sigma$-finite measures. 
\end{theorem}
\begin{proof}
Note that the corresponding result when $\nu,\nu'$ are
$\sigma$-finite kernels (in 
particular, finite kernels) is a special case of
\citep[Theorem 1.28]{kallenberg2017random}.
We will use this result to generalise it.

For the section-wise existence statement, fix $x\in X$ and apply the argument
below to the measures $\nu(x)$ and $\nu'(x)$, i.e. to kernels with source $1$.
Choosing one resulting derivative for each $x$ gives $\dif\nu'/\dif\nu(x,-)$
measurable in $Y$.
It remains to prove the joint measurability claim under the additional
hypothesis that $X$ is countable and discrete or that $Y$ is standard Borel.
Under this hypothesis, Theorem~\ref{thm:sfinalt} lets us write
$$
\nu=\kact{\mu}{f},
$$
where $\mu$ is a $\sigma$-finite kernel and
$f:X\times Y\to\{1,\infty\}$ is jointly measurable.
For each $x$, write $A^x:=\{y\in Y\mid f(x,y)=\infty\}$; by the proof of
Theorem~\ref{thm:topinftysets}, $A^x$ is a top $0$-$\infty$-set for $\nu(x)$.

Since $f\geq 1$, $\mu(x,B)=0$ implies $\nu(x,B)=0$ for all $x$ and $B$.
Thus $\nu'\ll\nu$ implies $\nu'\ll\mu$.
Decompose $\nu'$ as a countable sum of finite kernels and apply the classical
Radon-Nikod\'ym theorem for $\sigma$-finite kernels to each summand; by monotone
convergence, this yields a jointly measurable function
$r:X\times Y\to[0,\infty]$ such that
$$
\nu'=\kact{\mu}{r}.
$$

We claim that, for each $x$, the set
$$
E^x:=\{y\in A^x\mid 0<r(x,y)<\infty\}
$$
is $\mu(x)$-null.
If not, then by $\sigma$-finiteness of $\mu(x)$ and by decomposing $E^x$ into
level sets $\{1/m\leq r(x,-)\leq m\}$, there would be a measurable
$B\subseteq E^x$ such that
$$
0<\int_B \mu(x,\dif y)\,r(x,y)<\infty.
$$
But $B\subseteq A^x$, and $A^x$ is a $\nu(x)$-$0$-$\infty$-set; since
$\nu'\lli\nu$, the set $B$ must be a $\nu'(x)$-$0$-$\infty$-set.
This contradicts
$$
\nu'(x,B)=\int_B \mu(x,\dif y)\,r(x,y)\in(0,\infty).
$$

Define $h:X\times Y\to[0,\infty]$ pointwise by
$$
h(x,y)=
\begin{cases}
r(x,y), & f(x,y)=1,\\
1, & f(x,y)=\infty\text{ and }r(x,y)=\infty,\\
0, & \text{otherwise.}
\end{cases}
$$
Then $h$ is jointly measurable under the standard-Borel or countability
hypothesis above.
Moreover, $f(x,-)h(x,-)=r(x,-)$ $\mu(x)$-almost everywhere, by the claim and
our convention that $\infty\cdot0=0$.
Therefore
$$
\kact{\nu}{h}=\kact{\mu}{fh}=\kact{\mu}{r}=\nu',
$$
so $h$ is the required jointly measurable Radon-Nikod\'ym derivative of $\nu'$
with respect to $\nu$ under the additional hypothesis.

For a counterexample when $\nu$ is s-finite and we only require 
$\nu'\ll\nu$, take $\nu=\sum_{n\in\naturals}\uniform{\reals}$ and 
$\nu'=\mathsf{normal}(0,1)$.
In that case, $\nu'\ll\nu$.
However, for any $f:\reals\to[0,\infty]$, $A\in\Sigma_{\reals}$, we have 
$\kact{\nu}{ f}(A)\in\{0,\infty\}$, which means we can never have 
$\kact{\nu}{ f}=\nu'$.
\end{proof}

One reason the Radon-Nikod\'ym Theorem is important is that it implies the 
existence of a general \emph{importance sampling} procedure for arbitrary 
probabilistic programs (which \citep{staton2017commutative} has shown have 
semantics given by s-finite kernels).

Indeed, let $\mu\lli \nu$ be s-finite kernels $X\leadsto Y$,
where $X$ is countable and discrete or $Y$ is standard Borel.
Then, we can construct a jointly measurable RN-derivative
$\dif\mu/\dif \nu:X\times Y\to [0,\infty]$.
This has the property that $\kact{\nu}{ \dif\mu/\dif\nu} =\mu$.
In the computational terms of the model probabilistic programming language of
\citep{staton2017commutative}, this gives us the importance sampling procedure 
for $\mu$ with respect to $\nu$:
\[
\fbox{\mbox{$
\sample(\mu(x)) $
}} = \fbox{\mbox{$
\letin{y}{\sample(\nu(x))}{
\score(\frac{d\mu}{d\nu}(x,y));y}$}}
\]

It also recovers the usual \emph{rejection sampling} procedure in the
normalised case.
Indeed, suppose in addition that $\mu$ and $\nu$ are probability kernels and
that $\dif\mu/\dif\nu$ is bounded by some $M\in[1,\infty)$.
Then, we get a rejection sampling procedure for $\mu$:
\[
\fbox{\mbox{$\sample(\mu(x))$}} = 
\fbox{\mbox{$\letrecin{f}{\lambda \_.
\begin{array}{l}
\letin{y}{\sample(\nu(x))}{}\\
\letin{w}{\sample(\uniform{[0,1]})}{}\\
\ifelse{w\leq\frac{1}{M}\frac{d\mu}{d\nu}(x,y)}{y}{f()}
\end{array}
}{f()}$}}.
\]

\section{Lebesgue Decomposition}
While the Radon-Nikod\'ym theorem is a powerful tool for comparing kernels,
it does not typically apply to an arbitrary pair of kernels $k,l$, since usually
$k\lli l$ does not hold.
The Lebesgue decomposition theorem gives an analysis of the relationship
between an arbitrary pair of well-behaved kernels.
It is typically phrased for $\sigma$-finite measures and appears as \citep[Theorem
1.28]{kallenberg2017random} for $\sigma$-finite kernels.
In this section, we generalise the result to s-finite kernels.

To do this, we introduce a new relationship $k\boti l$ between kernels
$X\leadsto Y$
(read ``$k$ is \emph{$\infty$-singular} w.r.t.\ $l$'').
Write $k\boti l$ iff $k\ll l$, all $k(x)$-$0$-$\infty$-sets are trivial (i.e.\ are null
sets) for every $x\in X$, and there exists an $A\in\Sigma_{X\times Y}$ such
that $k$ is supported on $A$ and $A^x$ is an $l(x)$-$0$-$\infty$-set for every
$x\in X$.

\begin{theorem}[Lebesgue Decomposition Theorem for S-finite Kernels]
\label{thm:lebesguedecomp}
Let $k,l:X\leadsto Y$ be s-finite kernels where either $X$ is countable and discrete
or $Y$ is standard Borel.
Then, $k$ decomposes uniquely as a sum of three mutually singular components
$$
k=k_a+k_\infty+k_s,
$$
where $k_a\lli l$, $k_\infty\boti l$ and $\mutsing{k_s}{l}$. 
It then follows that $k_a,k_s$ are s-finite and $k_\infty$ is $\sigma$-finite.
\end{theorem}
\begin{proof}
Note that the special case of this theorem when $k,l$ are $\sigma$-finite appears as
\citep[Theorem 1.28]{kallenberg2017random} (in this case $k_\infty=0$,
and $\lli$ and $\ll$ are equivalent because we are working
with $\sigma$-finite kernels).

By Theorem~\ref{thm:sfinalt}, we obtain measurable 
$f_k:X\times Y\to \{1,\infty\}$ and $f_l:X\times Y\to\{1,\infty\}$
and $\sigma$-finite kernels $k',l':X\leadsto Y$ such that
$k=\kact{k'}{ f_k}$
and
$l=\kact{l'}{ f_l}$.

Let us write $\infty[k]:=f_k^{-1}(\infty)$ and $\infty[l]:=f_l^{-1}(\infty)$
and observe that they define measurable subsets of $X\times Y$.

Because both $k'$ and $l'$ are $\sigma$-finite kernels,
we are in a position to apply 
\citep[Theorem 1.28]{kallenberg2017random} to $k'$ and $l'$ to obtain a $\sigma$-finite decomposition
$$
k'=k'_{a\infty}+k'_{s}
$$ 
where $k'_{a\infty}\ll l'$ (and therefore $k'_{a\infty}\lli l'$) and
$k'_{s}\bot l'$.

Define 
\begin{align*}
k_s&:= \kact{k'_s }{ f_k};\\
k'_\infty&:=\kact{k'_{a\infty}}{ \chi_{\infty[l]\setminus
\infty[k]}};\\
k_\infty&:=\kact{k'_\infty}{ f_k};\\
k'_a&:=k'_{a\infty}-k'_\infty;\\
k_a&:=\kact{k'_a}{ f_k}.
\end{align*}
Then, it follows immediately that $k=k_a+k_\infty+k_s$.
Observe that $k_a,k_s$ are s-finite by part 3 of Theorem~\ref{thm:charac-sfinite}.
Moreover, it suffices to show that $k'_a,k'_\infty$ and $k'_s$ are mutually
singular, as the former three kernels have density $f_k$ w.r.t.\ the latter three,
respectively.
Now, $k'_a,k'_\infty\ll k'_{a\infty}$ (densities $1-\chi_{\infty[l]\setminus
\infty[k]}$ and $\chi_{\infty[l]\setminus
\infty[k]}$, respectively) and $k'_{a\infty}\ll l'$ (by construction
of the Lebesgue decomposition) and $k'_s\bot l'$ (by construction of the
Lebesgue decomposition), so $k'_a,k'_\infty\bot k'_s$.
Moreover, by construction $k'_\infty$ is supported in
$\infty[l]\setminus \infty[k]$ while $k'_a$ is supported in $X\times Y\setminus
(\infty[l]\setminus \infty[k])$, so it follows that $k'_\infty\bot k'_a$.
We conclude that $k_a,k_\infty$ and $k_s$ are mutually singular.

We proceed to show that $k_s\bot l$, $k_a\lli l$ and $k_\infty\boti l$.

First, observe that $k_s\bot l$ as $k'_s\bot l'$, $k_s\ll k'_s$
(density $f_k$) and $l\ll l'$ (density $f_l$).

Second, $k_a\lli k'_a$ (density $f_k$), $k'_a\lli k'_{a\infty}$
(density $1-\chi_{\infty[l]\setminus \infty[k]}$) and $k'_{a\infty}\lli l'$
(by the construction of the Lebesgue decomposition, as observed above).
By transitivity of $\lli$, it follows that $k_a\lli l'$.
As $k_a$ is supported outside $\infty[l]\setminus \infty[k]$, by construction,
it then follows that $k_a\lli l$.

Third, $k_\infty$ is $\sigma$-finite: on its support
$\infty[l]\setminus\infty[k]$ we have $f_k=1$, so $k_\infty=k'_\infty$, and
$k'_\infty$ is a restriction of the $\sigma$-finite kernel $k'_{a\infty}$.
It immediately follows that all $k_\infty$-$0$-$\infty$-sets are trivial by
Theorem~\ref{thm:topinftysets}.
Moreover, $k_\infty\ll k'_\infty$ (density $f_k$), $k'_\infty\ll k'_{a\infty}$
(density $\chi_{\infty[l]\setminus\infty[k]}$), $k'_{a\infty}\ll l'$ (by
the construction of the Lebesgue decomposition) and $l'\ll l$ ($l$ has strictly positive
density $f_l$ w.r.t.\ $l'$).
By transitivity of $\ll$, it follows that $k_\infty\ll l$.
Finally, $k_\infty$ is supported in $\infty[l]$, by construction,
which is an $l$-$0$-$\infty$-set, so $k_\infty\boti l$.

It remains to show uniqueness of the decomposition.
For this purpose,
suppose also that $k=\kappa_a+\kappa_\infty+\kappa_s$ where $\kappa_a\lli l$,
$\kappa_s\bot l$ and $\kappa_\infty\boti l$ and
$\kappa_a,\kappa_\infty,\kappa_s$ are mutually singular.
Since these are mutually singular summands of the s-finite kernel $k$, they are
restrictions of $k$ to measurable supports, and hence are themselves s-finite.
In particular, $\kappa_\infty$ is $\sigma$-finite by
Theorem~\ref{thm:topinftysets}.

As $k_s\bot l$, we get $L^k\in \Sigma_{X\times Y}$ such that $l$ is supported in
$L^k$ and $k_s$ is supported in $X\times Y\setminus L^k$.
As $\kappa_s\bot l$, we get $L^\kappa\in \Sigma_{X\times Y}$ such that $l$ is
supported in $L^\kappa$ and $\kappa_s$ is supported in
$X\times Y\setminus L^\kappa$. 
Defining $L:=L^k\cap L^\kappa$, we get that $l$ is supported in $L$ and
$k_s$ and $\kappa_s$ are both supported in $K:=X\times Y\setminus L$.
Observe that, for $U\in \Sigma_{K^x}$, we have $l(x,U)=0$, and therefore
\begin{align*}
k_a(x,U)&=0\\
k_\infty(x,U)&=0\\
\kappa_a(x,U)&=0\\
\kappa_\infty(x,U)&=0
\end{align*}
as all these kernels are $\ll$ w.r.t.\ $l$ by assumption.
Now, by assumption
$$
k_a+k_\infty+k_s=k=\kappa_a+\kappa_\infty+\kappa_s,
$$
it follows that $k_s(x,U)=\kappa_s(x,U)$.
Furthermore, for $U\in \Sigma_{L^x}$, it is clear that $k_s(x,U)=0=\kappa_s(x,U)$, as $L$
is the support of $l$ while $k_s,\kappa_s\bot l$ by assumption.
We see that $k_s=\kappa_s$.

It remains to show that
$k_a=\kappa_a$ and $k_\infty=\kappa_\infty$ on $L$, where we already know that
$$
k_a+k_\infty= \kappa_a+\kappa_\infty,
$$
because $k_s=\kappa_s=0$.

To show that $k_a=\kappa_a$ and $k_\infty=\kappa_\infty$ on $L$, let us 
first restrict to $L\cap \infty[k]$.
Let $x\in X$ and $U\in\Sigma_{\infty[k]^x\cap L^x}$.
If $k(x,U)=0$, it follows immediately that $$k_a(x,U),\kappa_a(x,U),
k_\infty(x,U),\kappa_\infty(x,U)=0.$$
If $k(x,U)>0$, then $U$ is a non-trivial $k(x)$-$0$-$\infty$-set, as
$U\subseteq \infty[k]^x$.
If $k_\infty(x,U)>0$, then $\sigma$-finiteness of $k_\infty(x)$ and mutual
singularity of the summands would give a measurable $V\subseteq U$ with
$0<k(x,V)<\infty$, a contradiction. Therefore, $k_\infty(x,U)=0$ and,
similarly, $\kappa_\infty(x,U)=0$.
It also follows that $k_a(x,U)=\infty=\kappa_a(x,U)$.

It remains to show that $k_a=\kappa_a$ and $k_\infty=\kappa_\infty$ on
$L\setminus \infty[k]$.
Since $k$ is $\sigma$-finite there by Lemma~\ref{lem:kernelinftyset},
so are $k_a,\kappa_a,k_\infty,\kappa_\infty$, being summands of $k$.

Let us restrict our attention further to $L\cap \infty[l]\setminus \infty[k]$.
Let $x\in X$ and $U\in\Sigma_{(L\cap \infty[l]\setminus \infty[k])^x}$.
Here, we claim that $k_a(x,U)=\kappa_a(x,U)=0$.
Indeed, $k_a,\kappa_a\lli l$ by assumption, so if $k_a(x,U)>0$, then also
$l(x,U)>0$, but $U\subseteq \infty[l]^x$, so this implies that
$U$ is a non-trivial $l(x)$-$0$-$\infty$-set and hence also a $k_a(x)$-$0$-$\infty$-set,
which contradicts the fact that $k_a$ is $\sigma$-finite on
$L\setminus \infty[k]$.
We see that $k_a(x,U)=0$ and, similarly, $\kappa_a(x,U)=0$ and therefore
$k_\infty(x,U)=\kappa_\infty(x,U)$.

The final case is $k_a=\kappa_a$ and $k_\infty=\kappa_\infty$ on
$(L\setminus \infty[k])\setminus \infty[l]$.
Let $x\in X$ and $U\in\Sigma_{((L\setminus \infty[k])\setminus \infty[l])^x}$.
Then, $k_\infty(x,U)=0=\kappa_\infty(x,U)$ as $k_\infty,\kappa_\infty\boti l$
so their supports may be chosen, up to $l$-null sets, inside $\infty[l]$.
It also follows that $k_a(x,U)=\kappa_a(x,U)$.

This concludes our uniqueness proof.
\end{proof}
Now, by Theorem~\ref{thm:radon-nikodym}, we further know that $k_a$ above has a
measurable density $f:X\times Y\to [0,\infty]$ with respect to $l$.
Moreover, note that $k_\infty=0$ if $l$ is $\sigma$-finite.

\section{Disintegration}
It is a cornerstone of Bayesian inference that one can construct conditional
probability distributions: the posterior distribution arises as a particular
conditional distribution over the unobserved parameters, conditioned on the
observed parameters of the chosen statistical model.
However, conditional distributions are a notoriously subtle topic in a general
measure-theoretic setting.
A particularly clean treatment can be given using the notion of
\emph{disintegration}, which we shall treat in this section.
The so-called disintegration theorem establishes the existence of
conditional distributions in suitable circumstances.
The key result in this section will be a generalisation of this theorem to
s-finite kernels, which gives, in a sense, the precise conditions under which
Bayesian inference is possible for probabilistic programs.

\begin{definition}[Disintegration (Conditional Distributions)]
Let $\phi:X\to Y$ be a measurable function between two measurable spaces and 
let $\mu:Z\leadsto X$ and $\nu:Z\leadsto Y$.
We call a kernel $k: Z\times Y\leadsto X$ a \emph{disintegration} (or conditional
distribution) of $\mu$ with respect to $\phi$ and $\nu$ if
\begin{enumerate}
\item[$a)$] $(\delta^Z\otimes^r\nu);k=\mu$ (i.e.\ for all $z\in Z$,
$\nu(z);k(z,-)=\mu(z)$);
\item[$b)$] for all $z\in Z$, $k(z,y)$ is supported on $\phi^{-1}(y)$ for
$\nu(z)$-almost all $y$.
\end{enumerate}
\end{definition}
Note that if $k$ is a disintegration of $\mu$ with respect to $\phi$ and $\nu$,
this means in particular that $k(z,-)$ is a disintegration of $\mu(z)$
with respect to $\phi$ and $\nu(z)$ (in such a way that it is jointly
measurable in $Z$ and $Y$).

The typical case considered is where $\phi=\fst:X_1\times X_2\to X_1$,
$Z=1$, $\nu=\fst_*\mu$, $\mu$ is a probability measure, and $X_1,X_2$
are standard Borel spaces.
In this situation, the classical disintegration theorem tells us that a
disintegration $k$ always exists and may be chosen to be a probability kernel.
Here, we consider the more general case, since we shall be interested in
compositionally transforming probabilistic programs through disintegration
as discussed in \citep{shan2017exact}, which forces us to consider
disintegrations of kernels (representing open programs) and programs
that are not necessarily normalised (as subprograms might, in general, be
s-finite).

To understand the relevance of disintegration to Bayesian inference, consider
that a statistical model is generally specified as a \emph{joint}
probability measure
$\mu$ on a product space $X\times \Theta$, where $\Theta$ represents the space
of unobserved (latent) \emph{parameters} of the model and $X$ represents the
space of the observed \emph{data}.
In many cases, $\mu$ is specified as the pushforward of $\mu_1\otimes \mu_2$
along the coordinate-swap map $\Theta\times X\to X\times\Theta$, where
$\mu_1$ is a probability (or $\sigma$-finite, in the case of an improper prior)
measure on $\Theta$, usually referred to as the \emph{prior distribution},
and $\mu_2:\Theta\leadsto X$ is a probability kernel, usually referred to as the
\emph{likelihood}.
(Such a splitting of a joint distribution into a prior and a likelihood always
exists by the disintegration theorem, but is far from unique.
Ultimately, the joint distribution is what matters for most purposes.)
This lets us define the \emph{model evidence}
(also called the marginal likelihood, intuitively representing the probability
that the joint model assigns to different observations)
$\nu:=\fst_*\mu$, which is a probability
measure on the data space $X$.
The classical disintegration theorem now tells us that there is a disintegration
$x\mapsto \mu(-\mid\fst=x)$ in the form of a kernel $X\leadsto \Theta$
of $\mu$ with respect to $\fst$ and $\nu$.
We call this probability kernel a \emph{posterior distribution}.
Note that its defining property is that
$$
\nu\otimes\mu(-\mid \fst=-) = \mu = \mathsf{swap}_*(\mu_1\otimes\mu_2).
$$
This result (or one of its close cousins) is often referred to as
\emph{Bayes' theorem}.

The reader may have noted that we say \emph{a} posterior distribution above.
Indeed, disintegrations need not be unique.
As we shall see, $\mu(-\mid\fst=x)$ is unique for $\nu$-almost all $x$.
In many practical situations, $X=\reals^n$ and $\uniform{\reals^n}\ll \nu$,
which means that the posterior distribution is uniquely determined everywhere
if we assume it is a continuous kernel (rather than merely measurable).
(Indeed, sets of Lebesgue measure zero have empty interior.)
However, this is by no means always true, and posteriors
in general are not uniquely defined everywhere.

Now that the concept is sufficiently motivated, let us turn to proving the
uniqueness and existence properties of disintegrations.
First, we can easily establish the following uniqueness property of
disintegrations.
\begin{theorem}[Uniqueness of Disintegrations]
Let $\phi:X\to Y$ be a measurable function between standard Borel spaces,
let $\mu$ be a measure on $X$ and let $\nu$ be a 
measure on $Y$.
If $\nu$ is s-finite, then a disintegration of $\mu$ with respect to $\phi$ and
$\nu$ is $\nu$-almost everywhere 
$\infty$-uniquely determined.
\end{theorem}
\begin{proof}
Note that the special case of this statement where $\nu$ is $\sigma$-finite
appears as \citep[Theorem F.2.6]{pollard2002user}.
We show that this implies the general statement.

Indeed, observe that by Theorem~\ref{thm:sfinalt}, $\nu=\kact{\nu'}{ f}$
where $\nu'$ is $\sigma$-finite and $f:Y\to\{1,\infty\}$ measurable.
Then $k$ is a disintegration of $\mu$ w.r.t.\ $\nu$ and $\phi$ iff
$\kact{k}{\phi \dcompose f}$ is a disintegration of $\mu$ w.r.t.\ $\nu'$ and
$\phi$; property $b)$ is unchanged, and property $a)$ follows because $k$ is
supported on fibres of $\phi$.

We can apply \citep[Theorem F.2.6]{pollard2002user} to this equivalent
criterion; translating back along the density $f$ gives the required
$\nu$-almost everywhere $\infty$-uniqueness statement.
\end{proof}
Since a disintegration of kernels is, in particular, pointwise
a disintegration of measures, this uniqueness result also applies to
disintegrations of kernels.

We now turn to the existence of disintegrations.
\begin{theorem}[Disintegration Theorem] Let $X$ and $Y$ be standard Borel 
spaces, let $Z$ be a measurable space,
let $\mu:Z\leadsto X$ and $\nu:Z\leadsto Y$ be s-finite 
kernels, and let $\phi: X\to Y$ be a measurable function such that 
$\phi_* \mu \ll \nu$ and, for all $z\in Z$,
\footnote{
Equivalently, given the separate assumption $\phi_*\mu\ll\nu$, for all
$\nu(z)$-$0$-$\infty$-sets $U$, $\phi^{-1}(U)$ is a
$\mu(z)$-$0$-$\infty$-set.
The absolute-continuity assumption is still needed to ensure that preimages of
$\nu(z)$-null sets are $\mu(z)$-null.
This condition is strictly stronger than the requirement that $\phi_*\mu\lli\nu$
(as $\phi^{-1}(U)$ may have measurable subsets that are not of the form
$\phi^{-1}(V)$).
}
$\mu(z,\phi^{-1}(\infty[\nu(z)])\setminus \infty[\mu(z)])=0$.
Then, there exists an s-finite disintegration $k:Z\times Y\leadsto X$ of $\mu$
with respect to $\phi$ and $\nu$.

$k$ can be chosen to be a probability kernel if $\mu$ is $\sigma$-finite and
$\phi_*\mu=\nu$.

This theorem may fail for s-finite $\mu,\nu$ if we only require 
$\phi_*\mu\lli\nu$.
\end{theorem}
\begin{proof}
Step 0. We use the standard disintegration theorem for $\sigma$-finite
kernels: if $X=Y\times X'$, $\phi=\fst$, $\mu:Z\leadsto X$ and
$\nu:Z\leadsto Y$ are $\sigma$-finite kernels, and $\fst_*\mu\ll\nu$, then
there exists a $\sigma$-finite disintegration $k$ of $\mu$ w.r.t.\ $\nu$ and
$\fst$.
Indeed, decompose $\mu$ into mutually singular finite kernels, apply the usual
finite regular-conditional-probability theorem to each piece, weight the
resulting conditionals by the Radon-Nikod\'ym densities of their marginals with
respect to $\nu$, and sum the resulting mutually singular finite kernels.
The proof proceeds by gradually generalising this result in two further steps.
The first of these is standard and is used for the disintegration theorem for
$\sigma$-finite measures in \cite[Appendix F]{pollard2002user}.
The final step is new.

Step 1. Thus the theorem holds in the case where $X=Y\times X'$, $\phi=\fst$,
and both $\mu:Z\leadsto X$ and $\nu:Z\leadsto Y$ are $\sigma$-finite kernels.

Step 2. We generalise this further to the case where $X$ is an arbitrary
standard Borel space and $\phi$ is an arbitrary measurable function
$X\to Y$.
Given this setup, observe that $\langle\phi,\id[X]\rangle_*\mu$ defines a
$\sigma$-finite kernel $Z\leadsto Y\times X$ (as $\langle \phi,\id[X]\rangle$ is
a measurable embedding) which concentrates on
$\{\langle \phi(x),x\rangle \mid x\in X\}$.
Moreover, $\fst_*(\langle\phi,\id[X]\rangle_*\mu)= \phi_*\mu\lli \nu$.
Therefore, we can obtain a disintegration $k$ of $\langle\phi,\id[X]\rangle_*\mu$
with respect to $\nu$ and $\fst$, using the disintegration theorem obtained in
Step 1. 
Then, observe that $\snd_*k$ is a disintegration of $\mu$ with respect to
$\nu$ and $\phi$, as $\nu(z);\snd_*k(z,-)=\snd_*(\nu(z);k(z,-))=
\snd_*(\langle\phi,\id[X]\rangle_*\mu(z))=\mu(z)$.
Moreover, $\snd_*k$ is s-finite as $k$ is $\sigma$-finite.

Step 3. We generalise this further to the general case where $\mu$
and $\nu$ are allowed to be s-finite kernels.
Consider this general case.
Apply Theorem~\ref{thm:sfinalt} to obtain $\sigma$-finite kernels
$\mu':Z\leadsto X$ and $\nu':Z\leadsto Y$ as well as measurable functions
$f_\mu:Z\times X\to \{1,\infty\}$ and $f_\nu:Z\times Y\to\{1,\infty\}$ such that
$\mu=\kact{\mu'}{ f_\mu}$ and $\nu=\kact{\nu'}{ f_\nu}$.
Observe that $\phi_*\mu'\ll\nu'$ as $\mu$ and $\nu$ have non-zero densities
w.r.t.\ $\mu'$ and $\nu'$, while $\phi_*\mu\ll \nu$ by assumption.
Moreover, vacuously,
$\mu'(z,\phi^{-1}(\infty[\nu'(z)])\setminus \infty[\mu'(z)])=0$, as both
$\mu'$ and $\nu'$ are $\sigma$-finite.
Therefore, we can apply the disintegration theorem obtained in Step 2 to obtain
an s-finite disintegration $k$ of $\mu'$ with respect to $\nu'$ and $\phi$.
We now claim that $\kact{k}{ f_\mu}$ (which is an s-finite kernel by
part 3 of Theorem~\ref{thm:charac-sfinite})
is a disintegration of $\mu$ with respect to $\nu$ and $\phi$.
It is immediate that $\kact{k}{ f_\mu}$ inherits property $b)$ of a
disintegration from $k$.
It remains to demonstrate property $a)$.
The crucial observation in the proof is that our assumption that 
$$\mu(z,\phi^{-1}(\infty[\nu(z)])\setminus \infty[\mu(z)])=0$$
implies that we can choose $f_\mu$ and $f_\nu$ such that
$$
(\id[Z]\times \phi) ; f_\nu \leq f_\mu \qquad(*). 
$$
Indeed, choose the representing functions so that their $\infty$-loci are top
$0$-$\infty$-sets for $\nu$ and $\mu$, respectively. The possible failure of
$(*)$ is then contained in
$(\id[Z]\times\phi)^{-1}(\infty[\nu])\setminus\infty[\mu]$, which is
$\mu(z)$-null for every $z$. On this set $f_\mu=1$, so it is also
$\mu'(z)$-null; changing $f_\mu$ from $1$ to $\infty$ there does not alter
$\mu=\kact{\mu'}{f_\mu}$ and gives $(*)$.
Noting this, we can compute, for $z\in Z$ and $U\in\Sigma_X$:
{\everymath{\displaystyle}
\begin{calculation}
(\nu(z) \dcompose \kact{k(z,-) }{ f_\mu(z,-)})(U)
\step{decomposition $\nu$, definition $\nu',f_\nu$}
(\kact{\nu'(z)}{ f_\nu(z,-)}
\dcompose
\kact{k(z,-)}{ f_\mu(z,-)}
)(U)
\step{definition of kernel composition and action on kernels $\kact{-}{-}$}
\int_Y\nu'(z,\dif y)f_\nu(z,y)\int_U k(z,y,\dif x)f_\mu(z,x)
\step{linearity of integration}
\int_Y\nu'(z,\dif y)\int_U k(z,y,\dif x)f_\nu(z,y)\cdot f_\mu(z,x)
\step{property $b)$, disintegration $k$}
\int_Y\nu'(z,\dif y)\int_U k(z,y,\dif x)f_\nu(z,\phi(x))\cdot f_\mu(z,x)
\step{by observation $(*)$ above and the fact that $f_\mu,f_\nu\in\{1,\infty\}$}
\int_Y\nu'(z,\dif y)\int_U k(z,y,\dif x) f_\mu(z,x)
\step{definition of kernel composition and action on kernels $\kact{-}{-}$}
\Big(\nu'(z) \dcompose \big( \kact{k(z,-)}{ f_\mu(z,-)} \big) \Big)(U)
\step{linearity of integration}
\big(
\kact{(\nu'(z) \dcompose (k(z,-)))}{ f_\mu(z,-)}
\big)(U)
\step{property $a)$, disintegration $k$}
\kact{\mu'(z)}{ f_\mu(z,-)}(U)
\step{decomposition $\mu$, definition $\mu',f_\mu$}
\mu(z,U)
\end{calculation}}
This establishes property $a)$ for the disintegration $\kact{k}{ f_\mu}$.

For the probability-kernel statement, first note that in the stated case
$\nu=\phi_*\mu$ is $\sigma$-finite. Indeed, since $\mu$ is $\sigma$-finite,
$\infty[\mu(z)]$ is $\mu(z)$-null, and the displayed assumption gives
$\mu(z,\phi^{-1}(\infty[\nu(z)]))=0$ for all $z$. Hence
$\nu(z,\infty[\nu(z)])=0$, so $\nu$ is $\sigma$-finite by
Theorem~\ref{thm:topinftysets}.
The standard probability-kernel form of the $\sigma$-finite disintegration theorem
therefore applies; for $Z=1$ this appears as
\citep[Exercise 5.3]{pollard2002user}, and the parameterised version is the
standard kernel form used in Step 0.

To see that the theorem may fail for s-finite kernels if we only require 
$\phi_*\mu\lli \nu$, take $Z=1$, $X=\reals$, $Y=1$, $\mu=\uniform{\reals}$.
Then, $\phi_*\mu=\infty\cdot \delta_*$.
Take $\nu=\phi_*\mu$.
Then, there is no disintegration, since for any candidate conditional measure
$$
1=\mu([0,1])\neq \int_1(\phi_*\mu)(\dif x)\mu([0,1]\mid \phi=x)\in\{0,\infty\}.
$$
\end{proof}

\begin{remark}[Radon-Nikod\'ym Derivatives as Disintegrations]
Observe that Radon-Nikod\'ym derivatives arise as a special case of
disintegrations where $\phi=\id[X]$.
Indeed, this gives a kernel $k$ such that $k(x)$ is supported in $\{x\}$,
hence $k$ is merely a function $Z\times X\to [0,\infty]$.
\end{remark}

\section{Randomising S-finite Kernels}
In this section, we show how all s-finite kernels can be constructed from 
simple building blocks.

First, recall a pivotal result from measure theory, the 
so-called Randomisation Lemma:
every probability kernel is constructible as a pushforward of the uniform 
probability measure on $[0,1]$.

\begin{lemma}[Randomisation {\citep[Lemma 2.22]{kallenberg2006foundations}}]
\label{lem:randomisation}
Let $\sigma$ be a probability kernel from a measurable space $S$ to a 
standard Borel space $T$. 
Then, there exists a measurable function $\dett(\sigma) : S \times [0, 1] \to 
T$ such that $\dett(\sigma)(s, -)_*(\uniform{[0,1]})= 
\sigma(s)$, for all $s\in S$,
where we write $\uniform{[0,1]}$ for the uniform distribution on $[0,1]$. 
Conversely, every kernel obtained this way is a probability kernel.
\end{lemma}

We can extend this result to subprobability kernels from $S$ to $T$ by noting 
that they are the same as probability kernels from $S$ to $T+1$.
Therefore, we get the following generalisation of the randomisation lemma.
\begin{lemma}[Randomisation for Subprobability Kernels]
\label{lem:subrandomisation}
Let $\sigma$ be a subprobability kernel from a measurable space $S$ to a 
standard Borel space $T$.
Then, there exists a measurable partial function (a.k.a.~a deterministic 
kernel) $\dett(\sigma) : S \times [0, 1] \rightharpoonup T$ such that $
\dett(\sigma)(s,-)_*\uniform{[0,1]} = \sigma(s)$, for all $s\in S$,
where we write $\uniform{[0,1]}$ for the uniform distribution on $[0,1]$.
Moreover, we can choose $\dett(\sigma)$ such that $\dett(\sigma)(s,p)=\bot$ 
whenever $\sigma(s)=0$.
Conversely, any kernel obtained this way is a subprobability kernel.
\end{lemma}

Finally, we note that, by Theorem~\ref{thm:charac-sfinite}, we can write any 
s-finite kernel $\sigma$ as a countable sum of subprobability kernels 
$\{\sigma_n\}_{n\in\naturals}$:
$$
\sigma = \sum_{n\in\naturals} \sigma_n = \int_\naturals\counting_\naturals(\dif n) \, 
\sigma_n.
$$
This gives us the following.
\begin{theorem}[Randomisation Lemma for S-finite Kernels] 
\label{thm:defsfinite}
Let $\sigma$ be an s-finite kernel from a measurable space $S$ to a standard 
Borel space $T$. 
Then, there exists a measurable partial function (deterministic kernel) 
$\dett(\sigma): S\times \naturals\times [0,1]\rightharpoonup T$ such that
\[
\sigma(s)=
{\dett(\sigma)(s,-)_*(\counting_\naturals\otimes 
\uniform{[0,1]}),}
\]  
for all $s\in S$.
Conversely, any kernel obtained this way is s-finite.
\end{theorem}
\begin{proof} 
Assume $\sigma = \sum_n \sigma_n$ for subprobability kernels $\sigma_n$. 
It is straightforward to see that the map $((s,n), U) \mapsto \sigma_n(s, 
U)$---call it $\sigma'$---is a kernel from $S\times\naturals$ to $T$, 
where $\naturals$ is equipped with the discrete $\sigma$-algebra, and 
$ S\times \naturals$ is equipped with the $\sigma$-algebra generated by the 
measurable rectangles.
Moreover, because each $\sigma_n$ is a subprobability kernel, so is $\sigma'$.

By the Randomisation Lemma for Subprobability Kernels, there exists a measurable
partial function $\dett(\sigma') : ( S\times \naturals ) \times [0, 1] \rightharpoonup T$ such 
that for each $n, s$ and $U$
\[
\sigma'((s,n), U) = \dett(\sigma') ((s,n), -)_* \uniform{[0, 1]} (U) = 
\uniform{[0,1]} \big(
\set{r \mid \dett(\sigma')((s,n), r) \in U}
\big)
\]
Since $\sigma \, (s, U) = \sum_n \sigma' ((s,n), U)$, we have
\[
\sigma \, (s, U)
= \counting_{\naturals} 
\otimes \uniform{[0, 1]} \big(
\set{(n, r) \mid \dett(\sigma') ((s, n), r) \in U}
\big)
\] 
For the converse, we merely have to note that $\counting_\naturals$ is 
s-finite, as are deterministic kernels, and that s-finite kernels are closed 
under composition.
\end{proof}

We have obtained a characterisation of s-finite kernels as precisely 
the class obtained by closing under kernel composition:
\begin{itemize}
\item deterministic kernels (measurable partial functions);
\item the Lebesgue measure $\uniform{[0,1]}$ on $[0,1]$;
\item the counting measure $\counting_\naturals$ on $\naturals$.
\end{itemize}
In fact, we can replace the measure space $([0,1],\uniform{[0,1]})$ with
$([0,1),\uniform{[0,1)})$ in the above (since $\{1\}$ has measure $0$),
and we can note that the measure space $(\naturals\times 
[0,1),\counting_\naturals\otimes \uniform{[0,1)})$ is isomorphic to 
$([0,\infty),\uniform{[0,\infty)})$ by the map $(n,p)\mapsto n+p$ (whose 
inverse is given by the pair $\pair{\floor}{\id{}-\floor}$).
This gives us the following.
\begin{theorem}[Randomisation Lemma 2 for S-finite Kernels] 
\label{thm:defsfinite2}
Let $\sigma$ be an s-finite kernel from a measurable space $S$ to a standard 
Borel space $T$.
Then, there exists a measurable partial function 
(deterministic kernel) $\dett(\sigma): S\times [0,\infty)\rightharpoonup T$ 
such that
$$\sigma(s)=\dett(\sigma)(s,-)_*( \uniform{[0,\infty)}),$$
for all $s\in S$.  
Conversely, any kernel obtained this way is s-finite.
\end{theorem}
We have obtained a characterisation of s-finite kernels as precisely 
the class obtained by closing under kernel composition:
\begin{itemize}
\item deterministic kernels (measurable partial functions);
\item the Lebesgue measure $\uniform{[0,\infty)}$ on $[0,\infty)$.
\end{itemize}

In fact, there are many equivalent ways of rephrasing the above lemma.
One more useful variation is obtained by noting that
$(\reals,\uniform{\reals})$
and $([0,\infty),\uniform{[0,\infty)})$ are isomorphic measure spaces, with the
isomorphism given by $[r\in [-n-1, -n)\mapsto r-3\cdot\floor(r)-1,
r\in [n, n+1)\mapsto r + \floor(r)]_{n\in \naturals}$ and its inverse.
\begin{theorem}[Randomisation Lemma 3 for S-finite Kernels] 
\label{thm:defsfinite3}
Let $\sigma$ be an s-finite kernel from a measurable space $S$ to a standard 
Borel space $T$.
Then, there exists a measurable partial function 
(deterministic kernel) $\dett(\sigma): S\times \reals\rightharpoonup T$ 
such that
$$\sigma(s)=\dett(\sigma)(s,-)_*( \uniform{\reals}),$$
for all $s\in S$.  
Conversely, any kernel obtained this way is s-finite.
\end{theorem}
Clearly, $\dett(\sigma)$ cannot be chosen to be total, in general.
For instance, that would imply that $\sigma(x,Y)$ is constant in $x$.
However, in some specific cases, it is possible.
\begin{theorem}[Total Randomisation]
Let $\sigma$ be an s-finite measure on a standard 
Borel space $T$.
Then, there exists a (total) measurable function 
$\dett(\sigma): \reals\to T$
such that
$$\sigma= \dett(\sigma)_*( \uniform{\reals})$$
iff $\sigma(T)=\infty$.  
\end{theorem}
\begin{proof}
Note that $\dett(\sigma)_*(\uniform{\reals})(T)=
\uniform{\reals}(\dett(\sigma)^{-1}(T))
= \uniform{\reals}(\reals)=\infty$ if $\dett(\sigma)$ is total.

For the converse, the assumption $\sigma(T)=\infty$ implies that $T$ is
nonempty.
Let $t_0\in T$.
Observe that we can apply Theorem~\ref{thm:defsfinite3} to obtain
a partial function $\dett(\sigma)': \reals\rightharpoonup T$ such that
$$\sigma=\dett(\sigma)'_*( \uniform{\reals}).$$
Now, define $\dett(\sigma)''(t):=t_0$ if $\dett(\sigma)'(t)=\bot$ and
$\dett(\sigma)''(t):=\dett(\sigma)'(t)$ otherwise.
Then, $\dett(\sigma)'': \reals\to T$ is a total measurable function.
Moreover, if $D:=\dett(\sigma)'^{-1}(T)$, then
$\uniform{D}$ defines an s-finite measure on $\reals$ and
$\dett(\sigma)''_*\uniform{D}=\sigma$.
Moreover, $\uniform{D}(\reals)=\infty$.
Let $I_{-1}:=\emptyset$ and, for $n\in\naturals$, define
$I_n=[-r_n,r_n]$, where
$r_n:=\inf \{r\in[0,\infty] \mid \uniform{D}([-r,r])\geq n+1\}$.
By continuity of $r\mapsto \uniform{D}([-r,r])$, the sets
$D_n:=(I_n\setminus I_{n-1})\cap D$ satisfy
$D=\biguplus_{n\in\naturals} D_n$ and $\uniform{D_n}$ is a probability measure
for all $n\in\naturals$.
Therefore, Lemma~\ref{lem:randomisation} gives us total measurable functions
$f_n:[0,1]\to \reals$ such that $(f_n)_*\uniform{[0,1]} = \uniform{D_n}$.
Now, define $g:\reals\to\reals$, $g=[r\in [-n/2-1/2,-n/2)\mapsto f_n(r+n/2+1),
r\in [n/2,n/2+1/2)\mapsto f_n(r-n/2)]_{n\in\naturals}$.
Then, $g_*\uniform{\reals}=\sum_{n\in \naturals} (f_n)_*\uniform{[0,1]}
=\sum_{n\in\naturals}\uniform{D_n}
= \uniform{D}$.
For $\dett(\sigma):\reals\to T$ defined by $\dett(\sigma):=
g;\dett(\sigma)''$, it follows that 
$\dett(\sigma)_*\uniform{\reals}=\sigma$, while
$\dett(\sigma):\reals\to T$ is a total measurable function.
\end{proof}
In particular, there exists a measurable function
$f:\reals\to\reals\times \reals$ such that
$\uniform{\reals}\otimes\uniform{\reals}=f_*\uniform{\reals}$.

One reason these results are relevant is that they illustrate that
one quickly obtains the whole class of s-finite kernels if one starts from
a rather limited set of primitives and closes them under composition.
In particular, even a probabilistic programming language with a set of
primitives that at first sight may seem rather limited is expressive enough
to construct all s-finite kernels: a call-by-value language with constants for
all measurable
functions, a random number generator $\mathsf{sample}\, \uniform{[0,1]}$
for drawing from $\uniform{[0,1]}$, and a $\mathsf{score}$ construct for enforcing soft constraints suffices.
Indeed, following the argument in \citep{staton2017commutative},
\begin{itemize}
\item we can
define all measurable partial functions using measurable functions and
$\mathsf{score}$;
\item by Lemma~\ref{lem:subrandomisation},
we can define all subprobability kernels using
$\mathsf{sample}\, \uniform{[0,1]}$ and measurable
partial functions,
so, in particular, we can define the Poisson distribution
$\mathsf{sample}\,\mathsf{poisson}(1)$;
\item using $\mathsf{score}$, we can define
$\mathsf{sample}\,\counting_\naturals$ using the
previously outlined importance sampling procedure with
$\mathsf{sample}\,\mathsf{poisson}(1)$,
which, by Theorem~\ref{thm:defsfinite}, then lets us construct all s-finite
kernels.
\end{itemize}
In fact, given that s-finite kernels are closed under composition 
(Theorem~\ref{thm:sfincomp}) and under soft constraints (part 3 of Theorem~\ref{thm:charac-sfinite}), it follows that these primitives define precisely
the class of s-finite kernels.

\section{S-Finite Kernels Between Quasi-Borel Spaces}
As argued in \citep{staton2017commutative},
one can obtain a perfectly good denotational semantics of first-order
(fine-grained call-by-value) probabilistic programming languages by interpreting
the types as standard Borel spaces, the pure terms (complex values) as
measurable functions and the effectful terms (computations) as s-finite kernels.
However, it is a classical result in measure theory that the category of standard
Borel spaces (or that of measure spaces) and measurable functions is not
cartesian closed \citep{aumann1961borel}.
This has led \citep{heunen2017convenient} to introduce a more general notion of
space called \emph{quasi-Borel spaces} to interpret higher-order probabilistic
programming languages.
In this section, we shall prove analogues of the Radon-Nikod\'ym and
Lebesgue decomposition theorems for quasi-Borel spaces.

Briefly, the category of quasi-Borel spaces is defined as the category of
concrete sheaves (in the sense of \citep{baez2011convenient}) on the category
of standard Borel spaces and measurable functions with countable measurable
covers as its Grothendieck topology.
This immediately shows that the category of quasi-Borel spaces is a Grothendieck
quasi-topos and, in particular, is complete, cocomplete and cartesian closed.

Concretely, we recall the definition of a quasi-Borel space.
\begin{definition}[\citep{heunen2017convenient}]
A \emph{quasi-Borel space (qbs)} is a set $X$ together with
a set of functions $M_X\subseteq X^\reals$ (called the \emph{random elements})
such that 
\begin{itemize}
\item[(const)] all constant functions are in $M_X$;
\item[(comp)] $M_X$ is closed under
precomposition with measurable functions on $\reals$;
\item[(sheaf)] $M_X$ is closed under countable measurable case distinctions:
if $\reals=\biguplus_{i\in\naturals}U_i$, where $U_i\in\Sigma_\reals$
and $\alpha_i\in M_X$ for all $i$, then
$[\alpha_i|_{U_i}]_{i\in\naturals}$ is in $M_X$.
\end{itemize}

A \emph{morphism} $f\colon X\to Y$ is a function that respects the structure,
i.e.~if $\alpha\in M_X$ then $\alpha;f\in M_Y$.
Morphisms compose as functions, and we have a category $\Qbs$.

A qbs $X$ is a \emph{subspace} of a qbs $Y$ if $X\subseteq Y$
and $M_X=\set{\alpha:\reals\to X~\mid~\alpha\in M_Y}$.
\end{definition}
As discussed above, the category $\Qbs$ turns out to be complete, cocomplete
and cartesian closed.
Moreover, we can compare qbses to measurable spaces, as follows.

\begin{theorem}[\citep{heunen2017convenient}]
We have an adjunction
$$
\Sigma: \Qbs \rightleftarrows \Meas : M,
$$
with $\Sigma\dashv M$,
which restricts to an adjoint equivalence from the full subcategory
$\Sbs\subseteq \Meas$ to\\ $M[\Sbs]\subseteq \Qbs$.
The adjunction is compatible with the forgetful functors
$\Qbs\to\Set$ and $\Meas\to \Set$.
The qbs structure $M_X$ on a measurable space is defined as
$$
M_{X}:=\Meas(\reals,X)
$$
and the measurable space structure $\Sigma_X$ on a qbs $X$ is defined as
$$
\Sigma_{X}:=\{U\subseteq X \mid \forall \alpha\in M_X.
\alpha^{-1}(U)\in\Sigma_\reals\}. 
$$
That is, $\Sigma_X$ is the final $\sigma$-algebra w.r.t.\ the random elements
$M_X$ (i.e.\ the largest $\sigma$-algebra such that the random elements are
measurable functions).
\end{theorem}

While qbs-morphisms play the role of measurable functions, we can also introduce
an equivalent notion of s-finite kernel for qbses.
In fact, these turn out to arise as the Kleisli morphisms for a commutative
monad $T$.
The idea is to use the randomisation lemma as a definition.
\begin{definition}[(Randomisable) S-finite Measure
\citep{ScibiorKVSYCOMHG18}]
For a qbs $X$, we can define an \emph{(randomisable) s-finite measure}
to be a triple $\langle \Omega, \mu, \alpha\rangle$ of a standard Borel space
$\Omega$, a morphism $\alpha\in \Qbs(\Omega, X)$
and an s-finite measure $\mu$ on $\Omega$.
We can define the integral of any qbs-morphism $f\in \Qbs(X,[0,\infty])$ w.r.t.
$\langle\Omega,\mu, \alpha\rangle$:
\begin{align*}
\int_X:\left(\sum_{\Omega\in \mathsf{ob}(\Sbs)}
\sum_{\mu\;\text{s-finite measure on } \Omega}
\Omega\Rightarrow X\right) &\to
(X\Rightarrow [0,\infty])\Rightarrow [0,\infty]\\
\langle \Omega, \mu,\alpha\rangle &
\mapsto \left(f\mapsto \int_\Omega \mu(\dif \omega)
f(\alpha(\omega))\right).
\end{align*}
Following the classical idea of Schwartz, this lets us identify
$\langle \Omega,\mu, \alpha\rangle$ and $\langle \Omega',\mu', \alpha'\rangle$
if they define the same integral operator in the sense that
$$
\int_X\langle \Omega,\mu, \alpha\rangle=
\int_X\langle \Omega',\mu', \alpha'\rangle
$$
and write $TX$ for the set of such equivalence classes
$[\Omega,\mu,\alpha]$.
Moreover, we can define the random elements
$$
M_{TX}:=\{\lambda r:\reals.[\Omega,k(r,-),\alpha(r,-)]
\mid k:\reals\leadsto \Omega\text{ s-finite kernel and }
\alpha\in\Qbs(\reals\times \Omega,X)\}
$$
to obtain a qbs $TX$.
\end{definition}

This definition gives us a straightforward way of making $T$ into a monad.
\begin{theorem}[Commutative Monad, \citep{ScibiorKVSYCOMHG18}]
$T$ is a commutative monad on $\Qbs$ under the monadic operations
inherited from the continuation monad
$$((-)\Rightarrow [0,\infty])\Rightarrow [0,\infty].$$
\end{theorem}

% Here, we make a couple of observations that will let us simplify the definition
% of $TX$ which will put us in a better place to analyse the obtained notion of
% s-finite kernel for qbses.
% \begin{theorem}
% For a qbs $X$, $\Sigma_{X}$ is the initial $\sigma$-algebra with respect to
% $\Qbs(X,[0,\infty])$ (that is, the smallest $\sigma$-algebra on $X$, such that
% all $f\in \Qbs(X,[0,\infty])$ are measurable).
% \end{theorem}
% \begin{proof}
% \mv{Do.
% Basically, because the adjunction restricts to an adjoint equivalence on
% standard Borel spaces.}
% \end{proof}

Next, we give a new, simplified presentation of $TX$ and establish some results
about the elements of $TX$.
Observe that, for $[\Omega,\mu,\alpha]\in TX$, the morphism
$\alpha\in \Qbs(\Omega,X)$ yields $\alpha\in \Meas(\Omega, \Sigma_X)$
by applying the functor $\Sigma$ and abusing notation by simply writing
$\Omega$ for $\Sigma_\Omega$, since $\Omega$ is a standard Borel space.
This means that $\alpha_*\mu$ defines a measure on $\Sigma_X$ in the usual
measure-theoretic sense.
Moreover, by Theorem~\ref{thm:sfincomp}, this is an s-finite measure.
\begin{theorem}
$$\int_X\langle \Omega,\mu,\alpha\rangle=
\int_X \langle \Omega',\mu',\alpha'\rangle$$
iff
$$
\alpha_*\mu = \alpha'_*\mu'
$$
as measures on $\Sigma_X$.
That is, the equality on elements $\langle \Omega,\mu,\alpha\rangle$ is simply
the equality on the measures $\alpha_* \mu$ on $\Sigma_X$.
\end{theorem}
\begin{proof}
The equality of elements $\langle \Omega,\mu,\alpha\rangle$ and 
$\langle \Omega',\mu',\alpha'\rangle$ 
of $TX$ is defined
to be the equality of all integrals 
$$\int_\Omega \mu(\dif \omega)
f(\alpha(\omega))=\int_{\Omega'} \mu'(\dif \omega)
f(\alpha'(\omega)),\qquad \forall f\in \Qbs(X,[0,\infty]).$$
Observe that $\Qbs(X,[0,\infty])=\Meas(\Sigma_X,[0,\infty])$ as $\Sigma\dashv M$,
where we abuse notation by writing $[0,\infty]$ for $M_{[0,\infty]}$.
Therefore, this equality also coincides with the equality of all integrals
$$\int_\Omega \mu(\dif \omega)
f(\alpha(\omega))=
\int_{\Omega'} \mu'(\dif \omega)
f(\alpha'(\omega)),\qquad \forall f\in \Meas(\Sigma_X,[0,\infty]).$$
Observe that the classical Frobenius reciprocity result from measure theory
\citep{schilling2017measures} tells us that
$$
\int_\Omega \mu(\dif \omega)
f(\alpha(\omega)) = \int_{\Sigma_X} \alpha_*\mu(\dif x) f(x).
$$
Therefore, the equality of elements of $TX$ is also equivalently characterised
as the equality of all integrals 
$$\int_{\Sigma_X} \alpha_*\mu(\dif x) f(x)=
\int_{\Sigma_X} \alpha'_*\mu'(\dif x) f(x),
\qquad \forall f\in \Meas(\Sigma_X,[0,\infty]).$$
However, as $f$ can be approximated by simple functions, this is precisely the
same as saying that
$$\alpha_*\mu=\alpha'_*\mu'$$
as measures on $\Sigma_X$.
\end{proof}
That is, we can think of $TX$ as the set of (s-finite) measures (in the usual
measure-theoretic sense) on $\Sigma_X$ obtained as a pushforward of
some s-finite measure $\mu$ on a standard Borel space $\Omega$ along a qbs
morphism $\alpha\in\Qbs(\Omega,X)$.
Similarly, the random elements of $TX$ are simply (s-finite) kernels
$\alpha_*k:\reals\leadsto \Sigma_X$
obtained by pushing forward some s-finite kernel
$k:\reals\leadsto \Omega$, where $\Omega$ is a standard Borel space,
along a qbs morphism $\alpha\in\Qbs(\Omega,X)$.
Applying Theorem~\ref{thm:defsfinite2} to randomise $\mu$ and
$k$, we now obtain the following corollary.
\begin{corollary} We have the following simplified description of $TX$ and
$M_{TX}$:
$$
TX=\{\beta_*\uniform{[0,\infty)} \in \textsf{s-finite measures}(\Sigma_X)
\mid \beta\in \Qbs([0,\infty), X+\{\bot\})\}
$$
where $\beta_*\uniform{[0,\infty)}$ denotes the partial pushforward that ignores
mass sent to $\bot$, and
$$
M_{TX}=\{\lambda r:\reals.\beta(r,-)_*\uniform{[0,\infty)}\in
\textsf{s-finite kernels}(\reals,\Sigma_X)
\mid \beta\in
\Qbs(\reals\times[0,\infty), X+\{\bot\})\}.
$$
\end{corollary}
This makes it immediately clear that $TX$ has the universal property of being
the categorical image (coequaliser of the kernel pair of) the map
\begin{align*}
\int_X:
[0,\infty)\Rightarrow (X+\{\bot\}) &\to
(X\Rightarrow [0,\infty])\Rightarrow [0,\infty]\\
\beta &
\mapsto \left(f\mapsto \int_{\Sigma_X}\beta_*\uniform{[0,\infty)}(\dif x)f(x)
\right)
\end{align*}
or, equivalently, the map (which is equal to the map above, by Frobenius 
reciprocity)
\begin{align*}
\beta &
\mapsto \left(f\mapsto \int_{[0,\infty)}\uniform{[0,\infty)}(\dif r)f(\beta(r))
\right),
\end{align*}
where the integrand is understood to be $0$ when $\beta(r)=\bot$.
We can also immediately apply Theorem~\ref{thm:defsfinite2}
to see that, at least for standard Borel spaces $X$, this reproduces our
usual notion of s-finite kernel, as every s-finite kernel between standard
Borel spaces is randomisable.
\begin{corollary}
If $X$ is a standard Borel space, then $TX$ consists precisely
of all s-finite measures
on $X$ and $M_{TX}$ consists of all s-finite kernels
$\reals\leadsto X$, in the measure-theoretic sense.
More generally, for standard Borel spaces $X,Y$, $X\Rightarrow TY$ consists
precisely of all s-finite kernels $X\leadsto Y$.
\end{corollary}

For $\mu,\nu\in TX$, where $X$ is a qbs, let us define $\mu\ll\nu,
\mu\lli\nu,\mu\bot\nu,\mu\boti\nu$ as we would for any two measures
$\mu,\nu$ on $\Sigma_X$.
Moreover, recall from \citep{ScibiorKVSYCOMHG18} that we can add elements of
$TX$ (simply as measures) and that $X\Rightarrow T1$ has a natural right
action on $TX$ (observing that $T1=[0,\infty]$).
This action $\kact{-}{-}$ coincides with the usual action $\kact{-}{-}$
of measurable functions $\Sigma_X\to[0,\infty]$ on s-finite measures on
$\Sigma_X$ once we
observe that $\Qbs(X,[0,\infty])=\Meas(\Sigma_X,[0,\infty])$ (abusing notation
and simply writing $[0,\infty]$ for $M_{[0,\infty]}$ and recalling that
$\Sigma\dashv M$).
This puts us in a position to state and prove analogues of the
Radon-Nikod\'ym
and Lebesgue decomposition theorems for (randomisable) s-finite measures on
a qbs $X$.
\begin{theorem}[Radon-Nikod\'ym for Qbses]
Let $X$ be a qbs and let $\mu,\nu\in TX$.
Then, there exists a qbs morphism $f\in \Qbs(X,T1)$ such that
$$
\mu=\kact{\nu}{ f}
$$
iff
$$
\mu\lli \nu.
$$
\end{theorem}
\begin{proof}
The existence direction is a special case of Theorem~\ref{thm:radon-nikodym},
and the converse follows from Lemma~\ref{lem:tri-abs-cts}, once we note that
\begin{enumerate}
\item we have
$$
\Qbs(X,T1)=\Qbs(X,[0,\infty])=\Meas(\Sigma_X,[0,\infty]);
$$
\item $TX$ is a subset of the s-finite measures on $\Sigma_X$;
\item the relation $\lli$ on $TX$ is simply that inherited from measures on
$\Sigma_X$;
\item the action $\kact{-}{-}$ of $\Qbs(X,T1)$ on $TX$ coincides with the
usual action $\kact{-}{-}$ of $\Meas(\Sigma_X,[0,\infty])$ on s-finite
measures on $\Sigma_X$.
\end{enumerate}
\end{proof}
Similarly, RN-derivatives on qbses inherit the same uniqueness properties from
those for measurable spaces in Theorem~\ref{thm:uniquern}.

\begin{theorem}[Lebesgue Decomposition for Qbses]
Let $\mu,\nu\in TX$.
Then, $\mu$ decomposes uniquely as a sum of three mutually singular components
$$
\mu=\mu_a+\mu_\infty+\mu_s,
$$
where $\mu_a,\mu_\infty,\mu_s\in TX$ and
$\mu_a\lli \nu$, $\mu_\infty\boti \nu$ and $\mu_s\bot \nu$. 
\end{theorem}
\begin{proof}
Since $\mu$ and $\nu$ are, in particular, s-finite measures on $\Sigma_X$,
we can appeal to Theorem~\ref{thm:lebesguedecomp} to obtain a unique
decomposition of $\mu$ as a sum of three mutually singular components
$$
\mu=\mu_a+\mu_\infty+\mu_s,
$$
where $\mu_a\lli \nu$, $\mu_\infty\boti \nu$ and $\mu_s\bot \nu$.

It remains to show that $\mu_a,\mu_\infty,\mu_s$ are also
elements of $TX$, rather than merely s-finite measures on $\Sigma_X$.
To see this, it suffices
to exhibit elements $\beta_a,\beta_\infty,\beta_s\in
\Qbs([0,\infty),X+\{\bot\})$ such that
\begin{align*}
\mu_a &= \beta_a{}_* \uniform{[0,\infty)}\\
\mu_\infty &= \beta_\infty{}_* \uniform{[0,\infty)}\\
\mu_s &= \beta_s{}_* \uniform{[0,\infty)}.
\end{align*}
Observe that there is some $\beta\in \Qbs([0,\infty),X+\{\bot\})$ such that
$\beta_*\uniform{[0,\infty)}=\mu$, since $\mu\in TX$.
Now, as $\mu_a,\mu_\infty,\mu_s$ are mutually singular measures on $\Sigma_X$,
we get a measurable partition of $X$ as $S_a\biguplus S_\infty\biguplus
S_s$ such that $\mu_a$ is supported in $S_a$, $\mu_\infty$ is supported in
$S_\infty$ and $\mu_s$ is supported in $S_s$.
Applying the functor $\Sigma$ (which as a left adjoint preserves coproducts),
we observe that $\beta$ is, in particular, a measurable function
$[0,\infty)\to \Sigma_X+\{\bot\}$.
Therefore, $R_a:=\beta^{-1}(S_a), R_\infty:=\beta^{-1}(S_\infty)$ and
$R_s:=\beta^{-1}(S_s)$ define measurable subsets of $[0,\infty)$.
Observe that the (finite) measurable case distinction
$\beta_i:=[R_i.\beta, [0,\infty)\setminus R_i. \lambda
r. \bot]$ defines an element of $\Qbs([0,\infty), X+\{\bot\})$ by the qbs
axioms (const) and (sheaf) and it follows immediately that
$\beta_i{}_*\uniform{[0,\infty)}=\mu_i$, where $i\in \{a,\infty,s\}$.
This concludes the proof.
\end{proof}

It is at present not clear to the authors whether or how a disintegration theorem
can be established for qbses, or whether the Radon-Nikod\'ym
and Lebesgue decomposition theorems can be generalised to s-finite kernels between qbses.

\bibliography{bib/references}
\bibliographystyle{plainnat}

\end{document}
%  LocalWords:  colimiting colimit LocalWords